\documentclass[12pt]{article}

\usepackage{amsmath,amssymb,amsthm}
\usepackage{color,framed,accents}

\normalsize

\newtheorem{theorem}{Theorem}[section]
\newtheorem{lemma}[theorem]{Lemma}

\newtheorem{corollary}[theorem]{Corollary}
\newtheorem{conjecture}[theorem]{Conjecture}
\numberwithin{equation}{section}

\let\eps=\varepsilon
\def\deltab{\delta} 
\def\rsets{{\mathcal{S}_r(n)}} 

\def\({\bigl(}   \def\){\bigr)}
\def\abs#1{\mathopen|#1\mathclose|} \let\card=\abs
\def\Abs#1{\bigl|#1\bigr|} 

\def\Hrd{{H_r(\boldsymbol{d})}}
\def\calHrd{{\mathcal{H}_r(\boldsymbol{d})}}
\def\dvec{{\boldsymbol{d}}}
\def\xvec{{\boldsymbol{x}}}
\def\yvec{{\boldsymbol{y}}}
\def\zvec{{\boldsymbol{z}}}
\def\thetavec{{\boldsymbol{\theta}}}
\def\deltavec{\boldsymbol{\delta}}
\def\gammavec{\boldsymbol{\gamma}}
\def\zerovec{\boldsymbol{0}}
\def\Xvec{\boldsymbol{X}}

\def\Zvec{\boldsymbol{Z}}
\def\W{\varGamma}

\newcommand{\medtilde}{\protect\accentset{\sim}}

\def\dfrac#1#2{\lower0.15ex\hbox{\large$\frac{#1}{#2}$}}
\def\E{{\mathbb{E}}}
\def\nicebreak{\vskip0pt plus50pt\penalty-300\vskip0pt plus50pt }

\def\st{\mathrel{:}}
\def\nfrac#1#2{{\textstyle\frac{#1}{#2}}}
\def\dfrac#1#2{\lower0.15ex\hbox{\large$\frac{#1}{#2}$}}
\let\originalleft\left
\let\originalright\right
\renewcommand{\left}{\mathopen{}\mathclose\bgroup\originalleft}
\renewcommand{\right}{\aftergroup\egroup\originalright}

\definecolor{byzantine}{rgb}{0.74, 0.2, 0.64}

\definecolor{forestgreen}{rgb}{0.13, 0.55, 0.13}

\def\V{\operatorname{\mathbb{V\!}}}
\def\deltamax{\delta_{\mathrm{max}}}
\def\diag{\operatorname{diag}}
\def\tr{\operatorname{tr}}
\def\Prob{\mathbb{P}} \let\Pr=\Prob
\def\Bin{\operatorname{Bin}}

\allowdisplaybreaks

\def\calQ{\mathcal{Q}}
\def\calR{\mathcal{R}}

\def\calD{\mathcal{D}}
\def\calB{\mathcal{B}}
\def\calT{\mathcal{T}}
\def\frakW{\mathfrak{W}}

\def\thetavec{\boldsymbol{\theta}}
\def\dvec{\boldsymbol{d}}
\def\betavec{\boldsymbol{\beta}}
\def\betavecstar{\boldsymbol{\beta}^\ast}

\def\Var{\operatorname{Var}}
\def\Cov{\operatorname{Cov}}
\def\Reals{{\mathbb{R}}}
\def\Complexes{{\mathbb{C}}}
\def\Naturals{{\mathbb{N}}}
\def\Integers{{\mathbb{Z}}}

\def\norm#1{\mathopen\|#1\mathclose\|}
\def\Norm#1{\bigl\|#1\bigr\|}
\def\maxnorm#1{\norm{#1}_{\rm max}}
\def\svec{\boldsymbol{s}}
\def\trans{^{\mathrm{t}}}
\def\Dmin{D_{\rm min}}
\def\Dmax{D_{\rm max}}

\title{Degree sequences of sufficiently dense\\ random uniform hypergraphs}

\author{
Catherine Greenhill\thanks{Research supported by the Australian Research Council, Discovery Project DP190100977.} \\
\small School of Mathematics and Statistics\\[-0.8ex]
\small UNSW Sydney\\[-0.8ex]
\small NSW 2052, Australia\\
\small \tt c.greenhill@unsw.edu.au\\
\and
Mikhail Isaev\footnotemark[\value{footnote}]\\
\small School of Mathematics\\[-0.8ex]
\small Monash University\\[-0.8ex]
\small Vic 3800, Australia\\
\small\tt mikhail.isaev@monash.edu 
\and
Tam\'{a}s Makai\footnotemark[\value{footnote}]\\
\small School of Mathematics and Statistics\\[-0.8ex]
\small UNSW Sydney\\[-0.8ex]
\small NSW 2052, Australia\\
\small \tt t.makai@unsw.edu.au\\
\and
Brendan D. McKay\footnotemark[\value{footnote}]\\
\small Department of Computing\\[-0.8ex]
\small Australian National University\\[-0.8ex]
\small Canberra, ACT 2601, Australia\\
\small\tt brendan.mckay@anu.edu.au
}

\date{16 May 2022}

\begin{document}

\maketitle

\begin{abstract}
We find an asymptotic enumeration formula for the number of simple $r$-uniform hypergraphs with a given degree sequence, when the number of edges is sufficiently large.
The formula is given in terms of the solution of a system of equations.
We give sufficient conditions on the degree
sequence which guarantee existence of a solution to this system. 
Furthermore, we solve the system
and give an explicit asymptotic formula when the degree sequence is close to regular.
This allows us to establish several properties of the degree sequence of
a random $r$-uniform hypergraph with a given number of edges.
More specifically, we compare the degree sequence of a random $r$-uniform
hypergraph with a given number edges to certain models involving sequences of binomial or
hypergeometric random variables conditioned on their sum.
\end{abstract}

\section{Introduction}\label{s:intro}

Hypergraphs are useful for modelling relationships between objects in
a complex discrete system, and can offer improvements over graph
models in areas such as ecology~\cite{golubski}, quantum computing~\cite{morimae}
and computer vision~\cite{purkait}.
A hypergraph $H=(V,E)$ consists of a finite set $V$ 
of vertices and a finite set $E$ of edges, where each edge
is a subset of the vertex set.  Here edges do not contain
repeated vertices, and there are no repeated edges.
A hypergraph is $r$-\emph{uniform} if every edge contains $r$ vertices.
We present an asymptotic enumeration formula for the number of $r$-uniform
hypergraphs with a specified degree sequence, where the degree of
a vertex is the number of edges containing it.
Our formula holds for $3\leq r\leq\frac12 n$ and
$nr^4\log n\ll d\leq \frac12\binom{n-1}{r-1}$, where $d$ is the average degree, under very weak
restrictions on how much the degrees can vary. By symmetry, the
ranges obtained by complementing the edge set and/or complementing
each edge are also covered.
Using this formula, we
establish some results on the degree sequence of a random $r$-uniform
hypergraph with a given number of edges, verifying a conjecture of
Kam{\v c}ev, Liebenau and Wormald~\cite{KLW} for our parameter range.

To be more precise, we must introduce some notation.
Let $[a]$ denote the set $\{1,2,\ldots, a\}$ for any positive integer $a$. 
For infinitely many natural numbers $n$, let
$r(n)$ satisfy $3\leq r(n)\leq n-3$ and let
$\dvec(n) = \(d_1(n),\ldots, d_n(n)\)$ be a sequence
of positive integers.  We simply write $r$ for $r(n)$, and similarly
for other notation.
We assume that for infinitely many $n$, 
\begin{equation}
\label{parity}
r \, \text{ divides } \, \sum_{j\in [n]} d_j.
\end{equation}
All asymptotics in the paper are as $n$ tends to
infinity, along values for which (\ref{parity}) holds.
Define $\calHrd$ to be the
set of simple $r$-uniform hypergraphs with vertex set
$V=\{1,2,\ldots, n\}$ and degree sequence $\dvec$.
Write $e(\dvec):= \frac{1}{r} \sum_{j\in [n]} d_j$ for the number of edges
and $d:= d(\dvec)= \frac{1}{n} \sum_{j\in [n]} d_j$ for the average degree.

Our first aim is to find an asymptotic expression for $\Hrd = |\calHrd|$ 
for degree sequences~$\dvec$ which are neither too dense nor too sparse.

Our approach to hypergraph enumeration is based on the complex-analytical method. 
The answer is expressed in terms of high-dimensional integrals  resulting from Fourier 
inversion applied to a multivariable generating function.  Then, these integrals are 
approximated using multidimensional variants of the saddle-point method; see Section~\ref{s:main-proof} for more details. In the context of combinatorial enumeration, this 
method was pioneered by McKay and Wormald in 1990~\cite{McKW90}. Since then, many other 
applications of this method have appeared; see for example~\cite{CGM,correlation-immune,MM}, 
and the many results cited in \cite{mother}.
In particular, Kuperberg, Lovett and Peled~\cite{Kuperberg} prove an asymptotic formula for the number of $r$-uniform $d$-regular hypergraphs on $n$ vertices which holds 
when the number of edges in the hypergraph and
its complement are each at least $n^c$ (which implies that $r>c$)
for some sufficiently large constant~$c$ which is not identified explicitly.

Recently, Isaev and McKay~\cite{mother} developed a general theory  based on complex 
martingales for estimating the high-dimensional integrals which arise from the 
complex-analytical method. In this paper, we apply tools from~\cite{mother} in the 
hypergraph setting.

For a survey of enumeration results for graphs with given degrees,
see Wormald~\cite{W-survey}. Here we discuss only $r$-uniform hypergraphs
with $r\geq 3$.
Dudek, Frieze, Ruci{\' n}ski and {\v S}ileikis \cite{DFRS} gave an asymptotic formula 
for the number of $d$-regular
$r$-uniform hypergraphs on $n$ vertices when $r\geq 3$ is constant, 
assuming that $d=o(n^{1/2})$. 
Building on~\cite{BG},
Blinovsky and Greenhill~\cite[Corollary~2.3]{BG2} gave an asymptotic formula for 
$\Hrd$ that holds when
the maximum degree $d_{\max}$ satisfies $r^4 d_{\max}^3 = o(nd)$.
These results were obtained using the switching method.

By adapting the `degree switching and contraction mapping' approach of~\cite{LW1,LW2},
Kam{\v c}ev, Liebenau and Wormald~\cite[Theorem~1.2]{KLW}
proved that the degree sequence of a 
randomly chosen $r$-uniform hypergraph with $m$ edges is closely related to a 
random vector with entries chosen from suitable independent binomial
distributions, conditioned on the entries of the vector having sum $nd$.
More precisely, they prove that the ratio of the probabilities of a particular
vector $\dvec$ in these two models is well-approximated by a 
simple function of $r$ and~$\dvec$.
We will restate their theorem as Theorem~\ref{thm:NAN1.2} below.
This result holds under some assumptions, namely that the degrees do not
vary too much, the edge size is not too large and the average degree is
at most a sufficiently small constant times $\frac 1r\binom{n-1}{r-1}$.
Kam{\v c}ev, Liebenau and Wormald also considered sparse degree sequences
in~\cite[Theorem~1.3]{KLW},
which subsumes the enumeration results of~\cite{BG,DFRS}.

Our second aim is to apply our enumeration formula to study the degree
sequence of random uniform hypergraphs with given degrees.  In particular,
we prove a companion result to~\cite[Theorem~1.2]{KLW} which allows larger edge
size, more edges and more variation between the degrees, when the average degree
is large enough.  Furthermore, we verify (for our range of parameters)
a conjecture made in~\cite{KLW}, showing
that vectors of independent hypergeometric random variables, conditioned on
having sum $nd$, closely match the degree sequence of a random
uniform hypergraph with $nd/r$ edges almost everywhere.

\subsection{Notation, assumptions and our general results}\label{s:assumptions}

Define the density $\lambda$ as a function of $n$, $r$ and the average degree
$d$ by
\begin{equation}\label{eq:density}
d = \lambda\, \binom{n-1}{r-1}.
\end{equation}

Write $\rsets$ to denote the set of all subsets of $[n]$ of size $r$.  
Given a vector $\betavec = (\beta_1,\ldots, \beta_n)\in\Reals^n$,
for all $W\in \rsets$ define
 \begin{equation}
\label{eq:lambdaW-def}
 	 \lambda_W(\betavec) := \frac{e^{\sum_{j\in W} \beta_j}}{1+e^{\sum_{j\in W} \beta_j}}. 
 \end{equation}
Note that $\lambda_W(\betavec)$ is the probability that the edge $W$ appears in the
\emph{$\beta$-model} for hypergraphs with given degrees,
see for example~\cite{stasi}.
Let $\lambda(\betavec)$ be the average values of the $\lambda_W(\betavec)$; that is,
\[
  \lambda(\betavec)  := \binom{n}{r}^{\!-1}\! \!\sum_{W\in \rsets} \lambda_W(\betavec).
\]
Observe that $\lambda_W(\betavec),\, \lambda(\betavec) \in (0,1)$.

\bigskip

Define the positive symmetric $n\times n$ matrix $A(\betavec)=(a_{jk})$ as follows: 
\begin{equation}\label{eq:A-def}
   a_{jk} :=  \begin{cases}
                         \,\dfrac{1}{2}\, \displaystyle\sum_{W\ni j} 
                                           \lambda_W(\betavec)(1-\lambda_W(\betavec)),
                         & \text{~for $j=k\in [n]$}; \\[2.5ex]
                         \,\dfrac{1}{2}\, \displaystyle\sum_{W\supset \{j,k\}}
                                           \!\!\lambda_W(\betavec)(1-\lambda_W(\betavec)),
                         & \text{~ for $j,k\in [n],\,\, j\neq k$}.
                  \end{cases}
\end{equation}
We use $|M|$ to denote the determinant of a matrix $M$. 

Let $\betavecstar\in\Reals^n$ be a solution to the system of equations 
\begin{equation}
	\label{exact}
	\sum_{W\ni j} \lambda_W(\betavecstar)  = d_j  \quad \text{ for $j\in [n]$.}
\end{equation}
Summing~\eqref{exact} over $j\in [n]$ gives 
\begin{equation}\label{eq:dslacomp}
 d = \frac{1}{n}\sum_{j\in [n]}d_j= \frac{r}{n} \sum_{W\in\rsets}\!\! \lambda_W(\betavecstar) 
 =\lambda(\betavecstar) \binom{n-1}{r-1}. 
\end{equation}
This shows that $\lambda(\betavecstar)$ equals the density $\lambda$ defined
in (\ref{eq:density}).  
Similarly, if we write $\lambda_W$ or $A$ without 
argument, we always mean that the argument is $\betavecstar$.

Our main enumeration result is the following.

\begin{theorem}
\label{thm:main}
Let $\dvec =\dvec(n)= (d_1,\ldots, d_n)$ be a degree sequence. 
Suppose that $r=r(n)$ satisfies $3\leq r\leq n-3$ and
\begin{equation}\label{mainineq}
    r^3 (n-r)^3 \log{n} \ll \lambda(1-\lambda) n \binom{n}{r}.
\end{equation}
Further assume that $\betavecstar=(\beta_1^\ast,\ldots, \beta_n^\ast)$ is a solution of 
\emph{(\ref{exact})} such that
\begin{equation}
\label{beta-range}
   \max_{j,k\in [n]} |\beta_j^\ast-\beta_k^\ast| = O\left(\frac{n}{r(n-r)}\right).
\end{equation}
	Let 
$\lambda_W=\lambda_W(\betavecstar)$ be defined as in \emph{(\ref{eq:lambdaW-def})},
for all $W\in \rsets$, and let
$A=A(\betavecstar)$ be defined as in \emph{(\ref{eq:A-def})}.
Then
\begin{align*}
\Hrd = \frac{r \(1+ O(\eps)\)}{2^n\, \pi^{n/2} \, |A|^{1/2}} \, \prod_{W\in \rsets} \bigl( \lambda_W^{-\lambda_W}\, (1-\lambda_W)^{-(1-\lambda_W)} \bigr)
\end{align*}
where
\[ \eps:=  \frac{r(n-r)n}{\lambda(1-\lambda)\binom{n}{r}} +
   \frac{\log^9{n}}{n^2}
  \biggl(\frac{r^3(n-r)^3}{\lambda(1-\lambda)\binom{n}{r}}\biggr)^{\!3/2} \!
   +n^{-\Omega(\log{n})} = o\( (\log n)^{-1}\).\]
The implicit constant in the $O(\eps)$ term depends only on the implicit
constant in \emph{(\ref{beta-range})}. 
\end{theorem}

The enumeration problem has two natural symmetries:  given a hypergraph, we may replace every edge by its complement, or we may take the complement of the edge set.  These symmetries show that for a given
degree sequence $\dvec$,
\begin{equation}
\Hrd = H_{n-r}(\dvec') = H_r(\medtilde\dvec) = H_{n-r}(\medtilde\dvec')
\label{H-symmetries}
\end{equation}
where
\begin{equation}
\label{eq:d-symmetries}
\begin{aligned}
\dvec' &:= \bigl(e(\dvec)-d_1,\ldots, e(\dvec)-d_n\bigr),\\
\medtilde\dvec &:= \left(\binom{n-1}{r-1}-d_1,\ldots, \binom{n-1}{r-1}-d_n\right),\\
\medtilde\dvec' &:= \left(\binom{n-1}{r}- e(\dvec) + d_1,\ldots, \binom{n-1}{r}- e(\dvec) + d_n\right).
\end{aligned}
\end{equation}
Using these symmetries, we may assume that 
\[
r\leq n/2 \quad \text{ and } \quad e(\dvec) \leq \dfrac{1}{2}\binom{n}{r}.
\]
When these inequalities are both satisfied, we say that $(r,\dvec)$ belongs to the 
\emph{first quadrant}.

The conditions in Theorem~\ref{thm:main} are invariant under these two symmetries.
It is true, but not obvious, that the  asymptotic formula for
$H_r(\dvec)$ is also invariant under these symmetries.
We prove this in Lemma~\ref{lem:symmetries} below.

\begin{lemma}
\label{lem:symmetries}
Suppose that $\betavecstar$ is a solution to \emph{(\ref{exact})}.
Let $\betavec'$, $\medtilde\betavec$, $\medtilde\betavec'$ be vectors
with entries $\beta'_j$, $\medtilde\beta_j$, $\medtilde\beta'_j$ defined in
the fourth row of Table~\ref{t:symmetries} for all $j\in [n]$.
Then $\betavec'$, $\medtilde\betavec$, $\medtilde\betavec'$ are solutions of \emph{(\ref{exact})} 
for the degree sequences $\dvec'$, $\medtilde\dvec$ and
$\medtilde\dvec'$ defined in \emph{(\ref{eq:d-symmetries})}, respectively.
Furthermore, the following relationships hold:
\begin{gather*} 
\lambda_{V\setminus W}(\betavec') = \lambda_W,\quad
   \lambda_{W}(\medtilde\betavec) = 1-\lambda_W,\quad
   \lambda_{V\setminus W}(\medtilde\betavec') = 1-\lambda_W\quad\quad
   \text{for all $W\in\rsets$};
\\[1ex]
 |A(\betavec')| = \Bigl(\frac{n-r}{r}\Bigr)^2\, |A(\betavecstar)|,\quad
   |A(\medtilde\betavec)| =  |A(\betavecstar)|,\quad
    |A(\medtilde\betavec')| = \Bigl(\frac{n-r}{r}\Bigr)^2\, |A(\betavecstar)|;
\\[1ex]
 \abs{\beta'_j-\beta'_k} = \abs{\medtilde\beta_j-\medtilde\beta_k}
= \abs{\medtilde\beta'_j-\medtilde\beta'_k} = \abs{\beta^\ast_j-\beta^\ast_k}\quad\quad
\text{for all $j,k\in[n]$}.
\end{gather*}
\end{lemma}

\medskip

For the reader's convenience, in Table~\ref{t:symmetries}
we summarise information about our parameters under these symmetries.
\medskip

\begin{table}[ht!]
\renewcommand{\arraystretch}{1.2}
\begin{center}
\begin{tabular}{|c|c|c|c|}
\hline
 $H_r(\dvec)$ & $H_{n-r}(\dvec')$ & $H_r(\medtilde\dvec)$ & $H_{n-r}(\medtilde\dvec')$ \\
\hline
$d_j$ & $e(\dvec)-d_j$ & $\binom{n-1}{r-1} - d_j$ & $\binom{n-1}{r} - e(\dvec) + d_j$\\
\hline
 $d$ & $\frac{n-r}{r} d$ & $\frac{1-\lambda}{\lambda} d$ & $\frac{(1-\lambda)}{\lambda} \frac{(n-r)}{r} d$\\
\hline
 $\beta_j^\ast$ & $\frac{1}{n-r} \left(\sum_{k\in [n]} \beta_k^\ast \right) - \beta_j^\ast$
    & $-\beta_j^\ast$ & $\beta_j^\ast - \frac{1}{n-r} \left(\sum_{k\in [n]} \beta_k^\ast \right)$\\
\hline
 $\lambda_W$ & $\lambda_{V\setminus W}(\betavec') = \lambda_W$ & $\lambda_W(\medtilde\betavec) = 1-\lambda_W$ &
           $\lambda_{V\setminus W}(\medtilde\betavec') = 1- \lambda_W$\\
\hline
$|A(\betavecstar)|$ & $\left(\frac{n-r}{r}\right)^2\, |A(\betavecstar)|$ &
 $|A(\betavecstar)|$ & $\left(\frac{n-r}{r}\right)^2\, |A(\betavecstar)|$ \\
\hline
\end{tabular}
\caption{This table shows how the degrees, average degree, solution to (\ref{exact}), values of the
lambda parameters with $W\in \rsets$, and determinant of the matrix, behave under the symmetries. 
\label{t:symmetries}}
\end{center}
\end{table}

It follows from (\ref{H-symmetries}) and Lemma~\ref{lem:symmetries} that it suffices to prove
Theorem~\ref{thm:main} when $(r,\dvec)$ belongs to the first quadrant.
In this case, using \eqref{eq:dslacomp}
the assumptions of Theorem~\ref{thm:main} become
\begin{equation}\label{eq:assumptions}
  3\le r\le \dfrac12 n,~~ nr^4 \log n \ll d \leq \dfrac{1}{2} \binom{n-1}{r-1} 
  ~~ \text{and} ~ \max_{j,k\in [n]} |\beta_j^\ast - \beta_k^\ast| = O\(r^{-1}\),
\end{equation}
and the error term becomes
\[ O\left(\frac{nr^2}{d} + \frac{r^{6} n \log^9 n}{d^{3/2}} + n^{-\Omega(\log n)}\right).\]
Here we use the fact that $\lambda(1-\lambda)\binom{n}{r}$ is a lower bound on the number of edges
of any hypergraph in $\mathcal{H}_r(\dvec)$ and its complement.
The following lemma provides sufficient conditions on $r$ and $\dvec$ which guarantee
the existence of solutions to (\ref{exact}).

\bigskip

\begin{lemma}
\label{lem:degree-sufficient} 
Let $(r,\dvec)$ belong to the first quadrant.
Assume that there exists $\varDelta \geq 0$  such that for all $j \in [n]$,
\[
	d e^{-\varDelta/r} \leq d_j \leq d e^{\varDelta/r}.
\]
Further, assume that one of the following two conditions hold:
\begin{itemize}\itemsep=0pt
\item[\emph{(i)}] $\varDelta \leq \varDelta_0$ for some sufficiently small constant $\varDelta_0>0$;
	\item [\emph{(ii)}] $ r d = o(1) \binom{n-1}{r-1}$, $r = o(n)$, and $\varDelta = \Theta(1)$.
\end{itemize}
Then  there exists  $\betavecstar$ satisfying \emph{(\ref{exact})} such that\/
$
	\max_{j,k\in [n]} |\beta^\ast_j-\beta^\ast_k|= O(\varDelta/r). 
$
\end{lemma}

Uniqueness is a feature of similar situations~\cite[Section 3.3.4]{Bishop}, but we have not found
a proof of uniqueness in our case in the literature.  For completeness we provide
a short proof. 

\begin{lemma}\label{lem:unique}
For a given degree sequence $\dvec$, the solution $\betavecstar$ to \eqref{exact}
 is unique if it exists.
\end{lemma}

Even though~\eqref{exact} doesn't have an explicit solution in general,
we can evaluate the formula
in Theorem~\ref{thm:main} accurately if we have a sufficiently
precise estimate of~$\betavecstar$.
Stasi, Sadeghi, Rinaldo, Petrovi{\' c} and Fienberg~\cite{stasi}
stated without proof a generalization of an algorithm of \cite{Chat2011} that gives geometric
convergence to $\betavecstar$ if it exists.
Though we didn't use the iteration from~\cite{stasi}, we will demonstrate how
the precision to which an estimate of~$\betavecstar$ satisfies~\eqref{exact}
can be used to validate the corresponding estimate of $H_r(\dvec)$.
Our example will be degree sequences that are not far from regular,
which will allow us to investigate the degree sequences of random hypergraphs.

For $j\in[n]$ define $\delta_j:=d_j-d$.
Define $\deltavec:=(\delta_1,\ldots,\delta_n)$ 
and $\deltamax:=\max\{\norm{\deltavec}_\infty,1\}$.
Also define $R_t:=\sum_{j=1}^n \delta_j\trans$ for $t\ge 0$ and note that
$R_1=0$.

Recall the definition of $\lambda$ from (\ref{eq:density}).
We will find it convenient to write some quantities in terms of the
parameter~$Q$, which is invariant under the symmetries of~\eqref{H-symmetries}:
\[
     Q := (1-\lambda)(n-r)\,d
         = \lambda(1-\lambda)\,\frac{r(n-r)}n\binom nr.
\]
We continue to use the error term of Theorem~\ref{thm:main},
which in terms of $Q$ is
\begin{equation}
\label{eq:eps}
     \eps  = \frac{r^2(n-r)^2}{Q} + \frac{r^6(n-r)^6\,\log^9 n}{n^{7/2}Q^{3/2}}
           + n^{-\Omega(\log n)}.
\end{equation}
Our criterion for being ``near-regular'' is 
\begin{equation}\label{nearreg}
  \deltamax = O(Q^{3/5} n^{-3/5} ),
\end{equation}
which in the first quadrant is equivalent to $\deltamax = O(d^{3/5})$.

\medskip

\begin{theorem}\label{thm:nearreg}
  If\/ $3\leq r\leq n-3$ and assumptions~\eqref{mainineq} and~\eqref{nearreg}
  hold, then
  \begin{align*}
 \Hrd &=\biggl(\frac{r(n-r)(n-1)^{n-1}}{2^n\, \pi^n\, Q^n}\biggr)^{\! 1/2}
\( \lambda^\lambda (1-\lambda)^{1-\lambda} )^{-\binom nr} \\
 &{\qquad}\times
 \exp\biggl( -\frac{(n-1)\,R_2}{2Q}  + \frac{n^2\, R_2}{4Q^2}
    + \frac{(1-2\lambda)(n-2r)n\,R_3}{6Q^2} - \frac{n^3\,R_4}{12Q^3} + O(\hat\eps)
 \biggr),
\end{align*}
where $\hat\eps:=\eps+ \deltamax n^{3/5}Q^{-3/5}$ and $\eps$ is defined in \emph{(\ref{eq:eps})}.
\end{theorem}

\subsection{Degree sequences of random uniform hypergraphs}\label{ss:near-reg}

Assumption~\eqref{nearreg} is weak enough to include
the degree sequences of random hypergraphs with high probability.
Following the notation of Kam\v cev, Liebenau and Wormald~\cite{KLW},
we define three probability spaces of integer vectors.
Formulas will be given in Section~\ref{s:degree-models}.
\begin{itemize}
  \item $\calD_r(n,m)$ is the probability space of degree sequences of
       uniformly random $r$-uniform hypergraphs with $n$ vertices and
       $m$ edges.
  \item $\calB_r(n,m)$ is the result of conditioning $n$ independent
      binomial variables $\Bin(\binom{n-1}{r-1},p)$ on having sum~$nd$.
      (This distribution is independent of~$p$.)
   \item Note that each component of $\calD_r(n,m)$ has a hypergeometric
      distribution. $\calT_r(n,m)$ is the result of conditioning $n$ independent
      copies of that distribution on having sum $nd$.
\end{itemize}

The most important previous result on the near-regular case was obtained
by Kam\v cev, Liebenau and Wormald~\cite{KLW}. All the overlap between
\cite[Theorem~1.2]{KLW} and Theorem~\ref{thm:main} occurs in
Theorem~\ref{thm:nearreg}, so we restate their theorem here.

\begin{theorem}[{\cite[Theorem.~1.2]{KLW}}]\label{thm:NAN1.2}
Fix $\varphi\in(\frac49,\frac12)$.
For all sufficiently small $c>0$ and every $C>0$, suppose that $3\le r<c n^{1/4}/\log n$,
$r^3 d^{1-3\varphi}<c$ and $\log^C n \ll d < \frac cr\binom{n-1}{r-1}$.
Let $\dvec$ be a degree sequence with mean $d$ and $\deltamax\le d^{1-\varphi}$.
Then
\[
   \Prob_{\calD_r(n,m)}(\dvec) = \Prob_{\calB_r(n,m)}(\dvec) \,
      \exp\biggl(  \frac{r-1}{2} - \frac{(r-1)R_2}{2(1-\lambda)(n-r)d} 
          + O(\eta) \biggr),
\]
where 
\begin{align*}
 \eta &:= \begin{cases}
  \,\displaystyle\frac{\log^2 n}{\sqrt{n}} + \frac{d^{2-4\varphi}}{n} + d^{1-3\varphi}, & \text{ if $r=3$;}\\[2ex]
  \,\displaystyle\frac{r^2\log^2 n}{\sqrt{n}} + (\lambda n + r)r^2 d^{1-3\varphi}, & \text{ if $r\geq 4$.}
\end{cases}\end{align*}
\end{theorem}

The conditions of Theorem~\ref{thm:NAN1.2} allow for much lower
average degree than Theorem~\ref{thm:main},
but at the cost of stricter upper bounds on the edge size, the number of
edges, and the variation between the degrees.

As can be seen, the relation between $\calD_r(n,m)$ and $\calB_r(n,m)$
becomes rapidly more distant as~$r$ increases.  Theorem~\ref{thm:nearreg} would allow
a statement for all~$r$, but we prefer a statement that is more easily
compared to Theorem~\ref{thm:NAN1.2}.
Note that our formula agrees with the expression given in
Theorem~\ref{thm:NAN1.2} if $r=o(n^{1/2})$, since then
 $((n-1)/(n-r))^{(n-1)/2} \sim e^{(r-1)/2}$.

\begin{theorem}\label{thm:BvsD}
Suppose that $3\le r\le cn$ and $0<\lambda < c$ for some fixed $c<1$.
If $d\gg r^4n\log n$ and $\deltamax=O(d^{3/5})$
then
\[
   \Prob_{\calD_r(n,m)}(\dvec) = \Prob_{\calB_r(n,m)}(\dvec)\;
      \Bigl(\frac{n-1}{n-r}\Bigr)^{\!(n-1)/2}
       \!\! \exp\biggl( - \frac{(r-1)R_2}{2(1-\lambda)(n-r)d} 
       + O(\bar\eps) \biggr).
\]
where $\bar\eps:=\eps+ \deltamax d^{-3/5}$ and $\eps$ is defined in \emph{(\ref{eq:eps})}.
\end{theorem}

As noted in~\cite{KLW}, one can expect $\calT_r(n,m)$ to be a
better match to $\calD_r(n,m)$, especially for large edge sizes. 
We prove this for the full range of our parameters.

 \begin{theorem}\label{thm:DvsT}
  If\/ $3\leq r\leq n-3$ and assumptions~\eqref{mainineq} and~\eqref{nearreg}
  hold, then
\begin{align*}
   \Prob_{\calD_r(n,m)}(\dvec) &= \Prob_{\calT_r(n,m)}(\dvec)\;
      \Bigl(\frac{n-1}{n}\Bigr)^{\!(n-1)/2}
       \! \exp\biggl( \frac{R_2}{2Q} 
       + O(\hat\eps)\biggr), \\
   &= \Prob_{\calT_r(n,m)}(\dvec)\,
        \exp\biggl( -\frac 12 + \frac{R_2}{2Q} + O(n^{-1}+\hat\eps)\biggr), 
\end{align*}
where $\hat\eps:=\eps+ \deltamax n^{3/5} Q^{-3/5}$ and $\eps$ is defined in \emph{(\ref{eq:eps})}.
\end{theorem}

Kam\v cev, Liebenau and Wormald~\cite{KLW} conjectured that $\calD_r(n,m)$ is asymptotically
equal to $\calT_r(n,m)$ almost everywhere.

\begin{conjecture}[{\cite{KLW}}]\label{conj:Trules}
Let $2\le r\le n-2$ and $\min\{m,\binom nr-m\}=\omega(\log n)$.
Then there exists a set $\frakW$ that has probability
$1-O(n^{-\omega(1)})$ in both $\calD_r(n,m)$ and $\calT_r(n,m)$,
such that uniformly for all $\dvec\in\frakW$,
\[
    \Prob_{\calD_r(n,m)}(\dvec) = \Prob_{\calT_r(n,m)}(\dvec)\,(1+o(1)).
\]
\end{conjecture}

We prove their conjecture for our range of parameters.

\begin{theorem}\label{thm:Tconj}
 If $3\leq r\leq n-3$ and assumption~\eqref{mainineq} holds, then
there exists a set $\frakW$ that has probability
$1 - n^{-\Omega(\log n)}$ in both $\calD_r(n,m)$ and $\calT_r(n,m)$, such that
uniformly for all $\dvec\in\frakW$,
\[
    \Prob_{\calD_r(n,m)}(\dvec) = \(1+O(\eps+n^{1/10}Q^{-1/10}\log n
       + n^{-1/2}\log^2 n)\) \Prob_{\calT_r(n,m)}(\dvec).
\]
\end{theorem}

\subsection{Structure of the paper}\label{ss:structure}

Having now stated our main results, we describe the overall structure of the paper.
In Section~\ref{s:main-proof}, we outline how
$H_r(\dvec)$ can be expressed as an $n$-dimensional integral and state the
lemmas which lead to its evaluation.
In Section~\ref{s:bounds} we prove some necessary bounds concerning the
quantities $\lambda_W(\betavec)$ and $A(\betavec)$, and then in
Section~\ref{s:evaluate} we apply them to evaluate the integral, 
completing the proof of our main enumeration result, Theorem~\ref{thm:main}.
In Section~\ref{ss:unique} we address existence and uniqueness of solutions to (\ref{exact}),
proving Lemma~\ref{lem:degree-sufficient} and Lemma~\ref{lem:unique}.
Section~\ref{s:nearreg} examines the near-regular case, proving Theorem~\ref{thm:nearreg}. 
Then in Section~\ref{s:degree-models} we prove our results about the degree sequence
of random uniform hypergraphs, as stated in
Section~\ref{ss:near-reg}.  Finally, Section~\ref{s:technical} contains several
technical proofs that have been deferred, 
including the proof of Lemma~\ref{lem:symmetries}.

Some of the calculations in this paper are rather tedious, particularly
in Sections~\ref{s:nearreg} and~\ref{s:degree-models}. We carried out
the worst of them first using the computer algebra package Maple and
later checked them by hand.  All infinite series are based on Taylor's
theorem and so have clear-cut truncation criteria.

\nicebreak

\section{Proof outline for Theorem~\ref{thm:main}}\label{s:main-proof}

We will take advantage of Lemma~\ref{lem:symmetries} to work in the
first quadrant, where the conditions of Theorem~\ref{thm:main}
are given by~\eqref{eq:assumptions}.

The number $H_r(\dvec)$ of simple $r$-uniform hypergraphs with degree sequence
$\dvec = (d_1,\ldots, d_n)$
can be expressed using a generating function, where the power of
variable $x_j$ gives the degree of vertex $j$ for $j\in [n]$.  Each $W\in\rsets$ will
contribute a factor of~$\prod_{j\in W} x_j$, if $W$ is an edge in the hypergraph,
or~1 if $W$ is not an edge. Using $[x_1^{d_1}\cdots x_n^{d_n}]$ to denote
coefficient extraction, this gives
\begin{align*} H_r(\dvec) &= [x_1^{d_1}\cdots x_n^{d_n}] 
   \prod_{W\in \rsets } \Bigl(1 + \prod_{j\in W} x_j\Bigr)\\
  &= \frac{1}{(2\pi i)^n}\, \oint\cdots\oint\,   
  \frac{\prod_{W\in \rsets }\(1 + \prod_{j\in W} x_j\)}{ \prod_{j\in [n]} x_j^{d_j+1}}\, d\xvec,
\end{align*}
using Cauchy's coefficient formula for the second line.  Each integral is over
a contour enclosing the origin.
Recalling that $\betavecstar$ is a solution of (\ref{exact}),
we choose the $j$th contour to be a circle of radius $e^{\beta^\ast_j}$,
for $j\in [n]$.  This choice leads to the expression
\begin{align}
H_r(\dvec) &= (2\pi)^{-n}\, \exp\Bigl(-\sum_{j\in [n]} \beta^\ast_j d_j\Bigr)\,
  \int_{-\pi}^\pi \cdots \int_{-\pi}^\pi \frac{\prod_{W\in \rsets}
   \(1 + \prod_{j\in W}  e^{\beta^\ast_j + i\theta_j}\)}
     {\exp\(i\sum_{j\in [n]} d_j\theta_j\)} \, d\thetavec \nonumber \\
 &= P_r(\betavecstar)\, 
  \int_{-\pi}^\pi \cdots \int_{-\pi}^\pi\, 
   \frac{\prod_{W\in \rsets} \left(
   1 + \lambda_W\(\exp\(i\sum_{j\in W}\theta_j\) - 1\)\right)}
   {\exp\(i\sum_{j\in [n]} d_j\theta_j\)} \, d \thetavec, \label{eq:number-of-hypergraphs}
\end{align}
where the factor in front of the integral is given by
\begin{equation}
\label{eq:Prb-def}
 P_r(\betavecstar):= (2\pi)^{-n}\, \exp\biggl(-\sum_{j\in [n]} \beta^\ast_jd_j \biggr)\, 
  \prod_{W\in\rsets} \Bigl( 1+ e^{\sum_{j\in W} \beta^\ast_j}\Bigr). 
    \end{equation}
Let $F(\thetavec)$ denote the integrand, that is,
\begin{equation}
\label{eq:F-def}
  F(\thetavec):= 
   \frac{\prod_{W\in \rsets} 
     \(1 + \lambda_W(\exp(i\sum_{j\in W}\theta_j) - 1)\)}
   {\exp\(i\sum_{j\in [n]} d_j\theta_j\)}. 
\end{equation}
As we will see in Lemma~\ref{lem:truncate}, our choice of $\betavecstar$
ensures that the linear term in the expansion of $\log F(\thetavec)$ vanishes.

The maximum
value of $|F(\thetavec)|$ is 1, which is achieved if and only
if $\sum_{j\in W}\theta_j\equiv 0\!\!\pmod{2\pi}$ for all $W\in \rsets$.
If this condition holds then all $\theta_j$ must be equal modulo $2\pi$, 
as can
be seen by considering two $r$-subsets $W$, $W'$ which differ in just
one vertex and observing that such a pair of subsets exists for any pair of vertices.
Hence there are precisely $r$ points where $F(\thetavec)$
is maximised in $(-\pi,\pi]^n$, namely $\thetavec^{(1)},\ldots, \thetavec^{(r)}$,
where for $t\in [r]$ the point 
$\thetavec^{(t)} = (\theta_1^{(t)},\ldots, \theta_n^{(t)})$
is defined by
\[ \theta^{(t)}_1 = \theta^{(t)}_2 = \cdots = \theta^{(t)}_n \equiv\frac{2\pi t}{r}\pmod{2\pi}.
\]

We will estimate the value of the integral first in the regions close to $\thetavec^{(t)}$, for some $t\in[r]$,  then for the remainder of the domain.
Write $U_n(\rho)$ for the ball of radius $\rho$ around the origin, with respect to the
infinity norm; that is,
\[ U_n(\rho):=  \bigl\{ \xvec\in\Reals^n \st |x_j| \leq \rho \text{ for } j\in [n] \bigr\},
\]
and for $\rho>0$ define the region $\calR(\rho)$ as
\begin{equation}\label{eq:R-def}
\calR(\rho):=U_n(\rho) \cap \biggl\{\thetavec\in \Reals^n: \biggl|\, \sum_{j\in [n]} \theta_j \biggr|\le n r^{-1/2} \rho \biggr\}.
\end{equation}

Evaluation of the integral proceeds by the following sequence of lemmas, whose
proof is deferred to Section~\ref{s:evaluate}.
The first two lemmas give an estimate of the value of the integral over $U_n(r^{-1})$, by providing an estimate over $\calR(d^{-1/2}\log{n})$ and  $U_n(r^{-1})\setminus\calR(d^{-1/2}\log{n})$ respectively.

\begin{lemma}\label{lem:inbox}
	If assumptions \emph{\eqref{eq:assumptions}} hold then
	\[
	  \int_{\calR(d^{-1/2}\log{n})} F(\thetavec)\, d\thetavec
	    = (1 + O(\eps))\,\frac{\pi^{n/2}}{|A|^{1/2}},
	 \]
	 where $\eps$ is given in \emph{(\ref{eq:eps})}. 
\end{lemma}

 \begin{lemma}\label{L:step1}
	If assumptions  \eqref{eq:assumptions}   hold then
	\[
	\int_{U_n(r^{-1}) \setminus \calR(d^{-1/2}\log{n})} |F(\thetavec)| \, d \thetavec
	=   n^{-\Omega(\log n)}\, \frac{\pi^{n/2}}{|A|^{1/2}}.  
	\]
\end{lemma}

\medskip

Define the regions $U^{(t)}$ for $t\in [r]$ by
\begin{equation}\label{eq:def-Ut}
  U^{(t)}:= \bigl\{ \thetavec^{(t)} + \thetavec \!\!\pmod{2\pi} \, : \, \thetavec\in U_n(r^{-1})\bigr\}. 
\end{equation}
Let $\mathcal{B}:= \cup_{t\in [r]} U^{(t)}$. Since $F(\thetavec^{(t)} + \thetavec) = F(\thetavec)$ for all $\thetavec\in U_n(\pi)$, each of the regions $U^{(1)},\ldots, U^{(r)}$ makes an identical contribution to the integral. 
Lemmas~\ref{lem:inbox} and~\ref{L:step1} imply that under assumptions \eqref{eq:assumptions} we have
\begin{equation}
	\int_{\mathcal{B}} 
	\, F(\thetavec)\,  d \thetavec = (1+O(\eps))\,\frac{r\, \pi^{n/2}}{|A|^{1/2}}.
	\label{eq:B-integral}
\end{equation}

The integral in the region $U_n(\pi)\setminus \mathcal{B}$ is approximated
in the next result.
\begin{lemma}\label{lem:reduce-to-J0}
	If assumptions \eqref{eq:assumptions} hold then
	\[
	\int_{U_{n}(\pi)\setminus \mathcal{B}} 
	\, |F(\thetavec)|\,  d \thetavec = 
	n^{-\omega(n)}\, \frac{\pi^{n/2}}{|A|^{1/2}}. 
	\]
\end{lemma}

\bigskip

Continuing with the proof of Theorem~\ref{thm:main}, by combining Lemma~\ref{lem:reduce-to-J0} and \eqref{eq:B-integral} we obtain 
\begin{equation}
	\int_{U_n(\pi)} 
	\, F(\thetavec)\,  d \thetavec = (1+O(\eps))\,\frac{r\, \pi^{n/2}}{|A|^{1/2}}.
	\label{eq:integral}
\end{equation}
We can express
$P_r(\betavecstar)$ in a more convenient form, as follows:
\begin{align}
 P_r(\betavecstar) &\stackrel{(\ref{exact})}{=} (2\pi)^{-n}\, \frac{\prod_{W\in\rsets}\,\( 1 + e^{\sum_{j\in W}\beta^\ast_j}\)}
  { \exp\Bigl(\sum_{j\in [n]} \beta^\ast_j \,\sum_{W\ni j} \lambda_{W}\Bigr)}\nonumber\\
  &= (2\pi)^{-n}\, \frac{\prod_{W\in\rsets}\,\( 1 + e^{\sum_{j\in W}\beta^\ast_j}\)}
  { \exp\Bigl(\sum_{W\in \rsets} \lambda_{W}  \,\sum_{j\in W} \beta^\ast_j\Bigr)}\nonumber\\
  &= (2\pi)^{-n}\, \prod_{W\in\rsets}\, \frac{ 1 + e^{\sum_{j\in W}\beta^\ast_j}}
  { \exp\( \lambda_W \, \sum_{j\in W} \beta^\ast_j\)}\nonumber\\
  &= (2\pi)^{-n}\, \prod_{W\in\rsets}\, \biggl(\frac{ 1 + e^{\sum_{j\in W}\beta^\ast_j}}
  { e^{ \sum_{j\in W} \beta^\ast_j}}\biggr)^{\!\lambda_W}
  \( 1 + e^{\sum_{j\in W} \beta^\ast_j}\)^{1-\lambda_W}\nonumber\\
  &\stackrel{(\ref{eq:lambdaW-def})}{=}  (2\pi)^{-n}\, \prod_{W\in \rsets} \bigl( \lambda_W^{-\lambda_W} (1-\lambda_W)^{-(1-\lambda_W)}\bigr). \label{factor-out-front}
\end{align}

The proof of Theorem~\ref{thm:main} in the first quadrant is completed by substituting \eqref{eq:integral}
and \eqref{factor-out-front} into \eqref{eq:number-of-hypergraphs}.
The full statement of Theorem~\ref{thm:main} then follows from Lemma~\ref{lem:symmetries}.

\section{Properties of $A$ and other useful bounds}\label{s:bounds}

We will need to analyse the behaviour of $\lambda_W(\betavec)$, $\lambda(\betavec)$ and
$A(\betavec)$, not only when $\betavec$ is a solution of (\ref{exact}), but more generally.
We also need
\[
 \varLambda(\betavec) := \binom{n}{r}^{\!-1} \!\! \sum_{W \in \rsets} \lambda_W(\betavec)(1-\lambda_W(\betavec)).
\]

Recall that the elements of $A(\betavec)$ are sums of terms of the form $\lambda_W(\betavec)(1-\lambda_W(\betavec))$.
We start by establishing bounds on $\lambda_W(\betavec)$ and $1-\lambda_W(\betavec)$.

\begin{lemma}\label{lem:ratio-assist}
Denote by $f:\Reals^r\to \Reals$ the function
\[f(\xvec)=\frac{e^{\sum_{j=1}^r x_j}}{1+e^{\sum_{j=1}^r x_j}}.\]
Let $\xvec$, $\yvec$ satisfy $|x_i-y_i|\le \deltab/r$ for some constant $\deltab\ge 0$,
 and define $p:=|\{j:x_j\neq y_j\}|$. Then
\[	e^{-\deltab\, p/r} 
	 \leq \frac{f(\xvec)}{f(\yvec)} \leq 
	e^{\deltab\, p/r},\quad 
	e^{-\deltab\, p/r}  \leq \frac{1-f(\xvec)}{1-f(\yvec)} \leq 
	e^{\deltab \, p/r}. 
\]
\end{lemma}

\begin{proof}
	First suppose that $p=1$ and without loss of generality assume $x_1\neq y_1$.
	Then if $y_1\le x_1$ we have
	\[ \frac{f(\xvec)}{f(\yvec)} = \frac{e^{x_1+X}}{1+e^{x_1+X}}\cdot \frac{1+e^{y_1+X}}{e^{y_1+X}}  
	\leq e^{x_1-y_1}
	\leq e^{\deltab/r},\]
	where $X=\sum_{j=2}^r x_j=\sum_{j=2}^r y_j$. 
Observe that $\frac{1+e^y}{1+e^x}\le e^{y-x}$ whenever $x\le y$. Therefore when $y_1> x_1$,
	\[ \frac{f(\xvec)}{f(\yvec)} = \frac{e^{x_1+X}}{1+e^{x_1+X}}\cdot \frac{1+e^{y_1+X}}{e^{y_1+X}}  
	\leq \frac{1+e^{y_1+X}}{1+e^{x_1+X}}
	\le e^{y_1-x_1}
	\leq e^{\deltab/r}.\]
	As $\xvec$ and $\yvec$ are arbitrary vectors in $\Reals^r$, by symmetry we also have
	\[\frac{f(\xvec)}{f(\yvec)}\ge e^{-\deltab/r}.\]
	Similarly, 
	\[
	\frac{1-f(\xvec)}{1-f(\yvec)}=\frac{1+e^{y_1+X}}{1+e^{x_1+X}}\le \max\{e^{y_1-x_1},1\}\le e^{\deltab/r}
	\quad \mbox{and} \quad \frac{1-f(\xvec)}{1-f(\yvec)}\ge e^{-\deltab/r}.
	\]
	
	For arbitrary $\xvec, \yvec$, let $\zvec_0,\ldots, \zvec_p$
	be a sequence of elements of $\Reals^n$ with $\zvec_0=\xvec$, $\zvec_{p}=\yvec$
	such that  $\zvec_j$ and $\zvec_{j-1}$ differ in only one coordinate for $j=1,\ldots,p$.
	Then
	\[ \frac{f(\xvec)}{f(\yvec)} = \prod_{j=1}^{p} 
	\frac{f(\zvec_{j-1})}{f(\zvec_{j})}, \qquad
	\frac{1-f(\xvec)}{1-f(\yvec)} = \prod_{i=1}^{p} 
	\frac{1-f(\zvec_{j-1})}{1-f(\zvec_{j})},
	\]
	and the statement follows as there are exactly $p$ factors. 
\end{proof}

We will apply this lemma in two slightly different scenarios. First we compare $\lambda(\betavec)$
to $\lambda(\widehat\betavec)$ for two different vectors $\betavec$ and $\widehat{\betavec}$.

\begin{lemma}
	\label{lem:lambdaW-ratios-different-beta}
	Let $\betavec$ and $\widehat{\betavec}$ satisfy $\max_{j\in [n]} |\beta_j-\widehat\beta_j|\leq \deltab/r$
	for some nonnegative constant~$\deltab$.
	Then
	\[e^{-\delta}\,\lambda(\widehat{\betavec})\le \lambda(\betavec) \le e^{\delta}\,\lambda(\widehat{\betavec}).\]
\end{lemma}

\begin{proof}
By Lemma~\ref{lem:ratio-assist} we have for each $W\in \rsets$ that $e^{-\delta}\lambda_W(\widehat\betavec)\le \lambda_W(\betavec)\le e^{\delta}\lambda_W(\widehat\betavec)$.
The result follows from the definition of $\lambda(\betavec)$.
\end{proof}

In the second application we consider the ratios of $\lambda_W(\betavec)$ and
$\lambda_{W'}(\betavec)$ for $W,W'\in\rsets$.

\begin{lemma}
	\label{lem:lambdaW-ratios}
	Let $\betavec$ satisfy $\max_{j,k\in [n]} |\beta_j - \beta_k|\leq \deltab/r$
	for some nonnegative constant~$\deltab$.
	Then
	\begin{align*} 
		e^{-\deltab\, (1-|W\cap W'|/r)} 
		& \leq \frac{\lambda_W(\betavec)}{\lambda_{W'}(\betavec)} \leq 
		e^{\deltab\, (1-|W\cap W'|/r)},\\
		e^{-\deltab\, (1-|W\cap W'|/r)} & \leq \frac{1-\lambda_{W}(\betavec)}{1-\lambda_{W'}(\betavec)} \leq 
		e^{\deltab \, (1-|W\cap W'|/r)} 
	\end{align*}
	for all $W,W'\in\rsets$. Hence
	\[
	e^{-\deltab} \leq \frac{\lambda_{W}(\betavec)}{\lambda(\betavec)} \leq 
	e^{\deltab} \quad \text{ and } \quad 
	e^{-2\deltab} \leq \frac{\lambda_{W}(\betavec)\(1-\lambda_W(\betavec)\)}{\varLambda(\betavec)} \leq 
	e^{2\deltab}
	\]
	for all $W\in\rsets$.
\end{lemma}

\begin{proof}
Note that the $\beta_j$ terms corresponding to $j\in W\cap W'$ appear in both $\lambda_W(\betavec)$ and $\lambda_{W'}(\betavec)$.
Together with Lemma~\ref{lem:ratio-assist} this implies the first half of the statement. The bounds involving $\lambda(\betavec)$ and $\varLambda(\betavec)$ follow from the definitions of these quantities.
\end{proof}

We use the previous result to deduce that $\lambda(\betavec)$ and $\varLambda(\betavec)$ have the same order of magnitude when $\lambda(\betavec)$ is small enough. 

\begin{lemma}\label{lem:dlacomp}
	Let $\betavec$ satisfy $\max_{j,k\in [n]} |\beta_j-\beta_k|\leq \deltab/r$
	for a given nonnegative constant $\deltab$. If $\lambda(\betavec)\le 7/8$ then	
	\[ \frac{e^{-\deltab}}{256}\, \lambda(\betavec) \le \varLambda(\betavec) \le \lambda(\betavec). 
	\]
\end{lemma}

\begin{proof}
	The upper bound holds as 
	\[\binom{n}{r} \varLambda(\betavec) = \sum_{W \in \rsets} \lambda_W(\betavec)(1-\lambda_W(\betavec))
	 \le \sum_{W \in \rsets} \lambda_W(\betavec) = \binom{n}{r}\lambda(\betavec).\]
	
	Now consider the set $S=\{W \in \rsets:\lambda_W(\betavec)>\frac{15}{16}\}$.
	First assume that $|S|\leq \frac{15}{16}\,|\rsets|$. Then
	\[\binom{n}{r} \varLambda(\betavec)\ge \sum_{W \in \rsets\setminus S} \lambda_W(\betavec)(1-\lambda_W(\betavec))\stackrel{L.\ref{lem:lambdaW-ratios}}{\ge}
	 \sum_{W \in \rsets\setminus S} e^{-\deltab}\, \lambda(\betavec) \dfrac{1}{16}\ge \frac{e^{-\deltab}}{256}\lambda(\betavec)\binom{n}{r}.\]
	On the other hand if $|S|> \frac{15}{16}\, |\rsets|$, then
	\[\lambda(\betavec) \binom{n}{r} = \sum_{W \in S} \lambda_W(\betavec)> \left(\dfrac{15}{16}\right)^2 \binom{n}{r} > \dfrac{7}{8}\binom{n}{r},\]
	contradicting our assumption.
\end{proof}

Now we turn to the matrix $A(\betavec)$ and establish that the diagonal entries
are relatively close to each other, and similarly for the off-diagonal entries.

\begin{lemma}\label{lem:A-entries-tight}
	Let $\betavec$ satisfy $\max_{j,k\in [n]} |\beta_j-\beta_k|\leq \deltab/r$
	for some nonnegative constant $\deltab$.
	Then the entries of $A(\betavec)=(a_{jk})$ satisfy
	\[ e^{-4\deltab/r} \leq \frac{a_{jk}}{a_{j'k'}} \leq e^{4\deltab/r},
	\qquad
	e^{-4\deltab/r} \leq \frac{a_{jj}}{a_{kk}} \leq e^{4\deltab/r}\]
	for any $j,k,j',k'\in [n]$ with $j\neq k$ and $j'\neq k'$.
	Furthermore,
	\begin{align*}
		\dfrac{1}{2}\, e^{-4\deltab/r} \varLambda(\betavec) \binom{n-2}{r-2} &\leq a_{jk}
		\leq 
		\dfrac{1}{2}\, e^{4\deltab/r} \varLambda(\betavec) \binom{n-2}{r-2},\\[1ex]  
		\dfrac{1}{2}\, e^{-4\deltab/r} \varLambda(\betavec) \binom{n-1}{r-1} &\leq  a_{jj}
		\leq 
		\dfrac{1}{2}\, e^{4\deltab/r} \varLambda(\betavec) \binom{n-1}{r-1}. 
	\end{align*}
\end{lemma}

\begin{proof}
	We start with the case when $j\neq k$ and $j'\neq k'$. 
	Let $S_{jk}=\{W\in \rsets:W\supset \{j,k\}\}$. Recall that
	\[
	a_{jk}=\dfrac{1}{2}\sum_{W\in S_{jk}}\lambda_{W}(\betavec)(1-\lambda_{W}(\betavec)) \quad \mbox{and} \quad a_{j'k'}=\dfrac{1}{2}\sum_{W'\in S_{j'k'}}\lambda_{W'}(\betavec)(1-\lambda_{W'}(\betavec)).
	\]
	Both $S_{jk}$ and $S_{j'k'}$ contain exactly $\binom{n-2}{r-2}$ elements. 
We will show that there exists a bijection $\zeta:S_{j,k}\rightarrow S_{j'k'}$ 
such that for every pair $(W,W')$ with $W'=\zeta(W)$, we have
	\[  e^{-4\deltab/r}\lambda_{W'}(\betavec)(1-\lambda_{W'}(\betavec)) 
	\le \lambda_{W}(\betavec)(1-\lambda_{W}(\betavec)) 
	\le e^{4\deltab/r}\lambda_{W'}(\betavec)(1-\lambda_{W'}(\betavec)). \]
	By Lemma~\ref{lem:lambdaW-ratios}, this follows if
$|W\cap \zeta(W)|\ge r-2$ for all $W\in S_{jk}$.
	
	We can assume that either $\{j,k\}\cap\{j',k'\}=\emptyset$ or $j=j'$.
	Now consider the function $b:V\to V$, which is the identity for every vertex in $V\setminus \{j,k,j',k'\}$ and switches $j$ with $j'$ and $k$ with $k'$. This function can be extended to a function $\zeta:S_{jk}\to S_{j'k'}$ by assigning to each set $W\in S_{jk}$ the set $\{b(j) : j\in W\}$. Clearly $b$ is a bijection and so is $\zeta$, and $|W\cap \zeta(W)|\geq r-2$ for all $W\in S_{jk}$, as required. 
		
	The remaining results follow as
	\[
	a_{jj}=\frac{1}{r-1}\sum_{\substack{k=1\\k\neq j}}^{n} a_{jk} \quad \mbox{and} \quad \varLambda(\betavec)=\frac{1}{n}\sum_{j=1}^n a_{jj},
	\]
completing the proof.
\end{proof}

We also establish an upper bound on the determinant of $A(\betavec)$.
It follows easily from
the Matrix Determinant Lemma (see for example~\cite[equation (6.2.3)]{meyer}) that 
for any real numbers $a,b$,
\begin{equation}
\label{eq:aIbJ}
 |aI + bJ| = a^{n-1}(a + bn)
\end{equation}
where $I$ is the $n\times n$ identity matrix and $J$ is the $n\times n$ matrix with
every entry equal to one.

\begin{lemma}\label{lem:detA}
	Let $\betavec$ satisfy $\max_{j,k\in [n]} |\beta_j-\beta_k|\leq \deltab/r$
	for some nonnegative constant $\deltab$.
	Then
	\[
	|A(\betavec)|
	    =\exp\left(O(n)\log\left(\varLambda(\betavec) \binom{n-1}{r-1}\right)\right).
	\]
\end{lemma}

\begin{proof}
	Note that for any $\xvec \in \Reals^n$ we have
	\begin{align*}
	\xvec\trans A(\betavec)\, \xvec
	 &=\dfrac{1}{2}\sum_{W\in \rsets}\! \lambda_W(\betavec)(1-\lambda_W(\betavec))\biggl(\,\sum_{j\in W}x_j\biggr)^{\!2} \\
	&{}\stackrel{L.\ref{lem:lambdaW-ratios}}{\le}
	 \dfrac{1}{2}e^{2\deltab}\varLambda(\betavec)\sum_{W\in \rsets}\biggl(\,\sum_{j\in W}x_j\biggr)^{\!2}\!=\xvec\trans A'\, \xvec,
	\end{align*}
	where $A'=\frac{1}{2}e^{2\deltab}\varLambda(\betavec) \Bigl(\binom{n-2}{r-1}I+ \binom{n-2}{r-2}J\Bigr)$.
	Therefore, by the min-max theorem, the $k$-th largest eigenvalue of $A(\betavec)$ is at most the $k$-th largest eigenvalue of $A'$.
	Since $A(\betavec)$ is positive semidefinite, all its eigenvalues are non-negative, implying that $|A(\betavec)|\le |A'|$. Using (\ref{eq:aIbJ}), we have
	\[
	|A'|=\exp\left(O (n)\log\left(\varLambda(\betavec) \binom{n-1}{r-1}\right)\right)
	\]
	and the result follows.
\end{proof}

\subsection{Inverting $A(\betavec)$}\label{s:diagonalisation}

Next we bound the entries of $A(\betavec)^{-1}$ and find a change of basis matrix $T$
which transforms $A(\betavec)$ to the identity matrix.
For $p\in \{1,2,\infty\}$, we use the notation $\norm{\cdot}_p$ for the standard 
vector norms and the corresponding induced matrix norms 
(see for
example~\cite[Section~5.6]{HJ}).  In particular,  for an $n\times n$
matrix $M = (m_{ij})$,
\[ 
  \norm{M}_1 = \max_{j\in [n]} \,\sum_{i\in [n]} |m_{ij}|,\qquad
   \norm{M}_\infty = \max_{i\in [n]} \, \sum_{j\in [n]} |m_{ij}|.
\]

The proof of this lemma is given in Section~\ref{s:proof-diaggeneral-r}.

\begin{lemma}\label{lem:diaggeneral-r}
	Let $\deltab$ be a nonnegative constant. 
	For every $\betavec$ such that $\max_{j,k\in [n]} |\beta_j - \beta_k|\leq \deltab/r$ the following holds.

	Let $A(\betavec)^{-1}=(\sigma_{jk})$ be the inverse of $A(\betavec)$.
	There exists a constant $C$, independent of $\deltab$, such that for $n\ge 16e^{4\deltab}$ we have
		\begin{equation}\label{sigmabound}
	|\sigma_{jk}|\le 
	\begin{cases}
        \,\displaystyle \frac{Ce^{35 \deltab}}{\varLambda(\betavec) \binom{n-1}{r-1}},  & \mbox{if } j=k; \\[3ex]
	\,\displaystyle \frac{Ce^{35 \deltab}}{\varLambda(\betavec) \binom{n-1}{r-1}n}, & \mbox{otherwise.}
	\end{cases}
	\end{equation}
	In addition, there exists a matrix $T=T(\betavec)$ such that $T\trans A(\betavec)\, T=I$ with
	\[
	\norm{T}_1, \norm{T}_\infty 
	 =O\biggl(\varLambda(\betavec)^{-1/2} \binom{n-1}{r-1}^{\!-1/2}\,\biggr).
	\]
	Furthermore, for any $\rho>0$ there exists 
$\rho_1,\rho_2=\Theta\Bigl(\rho\, \varLambda(\betavec)^{1/2}\, \binom{n-1}{r-1}^{1/2} \Bigr)$ such that
	\[
	T\( U_n(\rho_1)\) \subseteq \calR(\rho) \subseteq T\(U_n(\rho_2)\),
	\]
	where $\calR(\rho)$ is defined in~\eqref{eq:R-def}.
\end{lemma}

\section{Evaluating the integral}\label{s:evaluate}

In this section we prove Lemmas~\ref{lem:inbox}--\ref{lem:reduce-to-J0}.  
We have already seen that these lemmas establish Theorem~\ref{thm:main}.

Throughout this section we assume that \eqref{eq:assumptions} holds and thus $\lambda=\binom{n-1}{r-1}^{-1}d \le \nfrac{1}{2}$.
Therefore, by Lemma~\ref{lem:dlacomp}, for $\varLambda:=\varLambda(\betavecstar)$ we have
\begin{equation}\label{eq:dlacomp}
	\varLambda\, \binom{n-1}{r-1}=\Theta\left(\lambda \binom{n-1}{r-1}\right)=\Theta(d).
\end{equation} 

\subsection{Proof of Lemma~\ref{lem:inbox}}

First, we will estimate the integral of $F(\thetavec)$ over $\calR (d^{-1/2} \log{n})$.
For $\xi\in[0,1]$ and $x\in[-1,1]$, $\abs{\xi(e^{ix}-1)}$ is bounded
below~1 and the fifth derivative of $\log \(1 + \xi(e^{ix}-1)\)$ with respect
to $x$ is uniformly~$O(\xi)$.  Using the principal
branch of the logarithm in this domain, we have by Taylor's theorem that uniformly
\begin{equation}\label{eq:truncTaylor}
    \log \(1 + \xi(e^{ix}-1)\) = \sum_{p=1}^4 i^p c_p(\xi) \,x^p
         + O(\xi)\,\abs{x}^5,
\end{equation}
where the coefficients are 
\begin{align*} c_1(\xi) := \xi,\qquad
c_2(\xi)&:= \dfrac{1}{2}\xi(1-\xi),\qquad
c_3(\xi):= \dfrac{1}{6}\xi(1-\xi)(1-2\xi),\\
c_4(\xi)&:= \dfrac{1}{24}\xi(1-\xi)(1-6\xi + 6\xi^2).
\end{align*}

\begin{lemma}\label{lem:truncate}
	Let $\rho:=d^{-1/2}\, \log n$.
        Then, for $\thetavec\in U_n(\rho)$, we have
        \[
		\log F(\thetavec)  
		=  -   \thetavec\trans A \,\thetavec
		+ \sum_{p=3}^4\,
		\sum_{W\in\rsets}\!\!  i^p c_p(\lambda_W)\, \biggl(\, \sum_{j\in W} \theta_j\biggr)^{\!p}
		+  O\biggl( \frac{n  r^{4}\log^5 n}{d^{3/2}}\biggr). 
	\]
\end{lemma}

\begin{proof}
  Recall that $\lambda_W\in(0,1)$ for all $W$, and note that~\eqref{eq:assumptions}
  implies that $r\rho=o(1)$.
 Hence, recalling  (\ref{eq:F-def}), we can apply~\eqref{eq:truncTaylor}
for each $W\in\rsets$, taking $\xi=\lambda_W$ and $x=\sum_{j\in W}\theta_j$. 
The linear term of $\log F(\thetavec)$ (which includes terms from the denominator of $F(\thetavec)$), is
	\[  i \, \sum_{j\in [n]} \theta_j\, \biggl(\, \biggl(\,\sum_{W\ni j} \lambda_W \biggr) - d_j \biggr),\]
	which equals zero by (\ref{exact}). 
	In addition, for the quadratic term,
\[
\sum_{W\in\rsets}\dfrac{1}{2}\lambda_{W}(1-\lambda_{W})\biggl(\,\sum_{j\in W} 
\theta_j\biggr)^{\!2}
=\sum_{j,k\in [n]} \sum_{W \supset \{j,k\}} \dfrac{1}{2}\lambda_{W}(1-\lambda_{W})\theta_j\theta_k=\thetavec\trans\!  A \thetavec.
\]
Now $\lambda_W = O(\lambda)$ for all $W\in\rsets$, 
by Lemma~\ref{lem:lambdaW-ratios}, so the combined error term is
\[  O\biggl( \lambda \binom nr r^5 d^{-5/2}\log^5 n\biggr) 
 \stackrel{(\ref{eq:dlacomp})}{=} O\biggl( \frac{n  r^{4}\log^5 n}{d^{3/2}}\biggr).  \qedhere \]
\end{proof}

\bigskip

Recall that for a complex variable $Z$, the \emph{variance} is defined by
\[ \Var Z = \E |Z - \E Z|^2 = \Var \Re Z + \Var \Im Z,\]
while the \emph{pseudovariance} is
\[ \V Z = \E (Z - \E Z)^2 = \Var \Re Z - \Var \Im Z + 2i\, \Cov(\Re Z, \Im Z).\]
The following is a special case of \cite[Theorem~4.4]{mother}
that is sufficient for our current purposes.

\begin{theorem}\label{gauss4pt}
	Let $A$ be an $n\times n$ positive definite symmetric real matrix
	and let $T$ be a real matrix such that $T\trans\! AT=I$.
	Let $\varOmega$ be a measurable set and let 
	$f: \Reals^n\to\Complexes$ and $h:\varOmega\to\Complexes$ 
	be measurable functions.
	Make the following assumptions for some $\rho_1,\rho_2,\phi$:
	\begin{itemize}\itemsep=0pt
		\item[\emph{(a)}] $T(U_n(\rho_1))\subseteq \varOmega\subseteq T(U_n(\rho_2)),$
		where $\rho_1,\rho_2=\Theta(\log n)$.
		
				\item[\emph{(b)}] For $\xvec\in T(U_n(\rho_2))$,
				$2\rho_2\,\norm{T}_1\,\left|\dfrac{\partial f}{\partial x_j}(\xvec)\right|
				\le \phi n^{-1/3}\le\frac23$ for $1\le j\le n$ and
				\[4\rho_2^2\,\norm{T}_1\,\norm{T}_\infty\,
				\norm{H}_\infty
				\le \phi n^{-1/3},\]
		where $H = (h_{jk})$ is the matrix with entries defined by
		\[		
		     h_{jk}=\sup_{\xvec\in T(U_n(\rho_2))}\, 
		     \left| \frac{\partial^2 f}{\partial x_j\, \partial x_k}(\xvec) \right|.
				\]
		\item[\emph{(c)}] $\abs{f(\xvec)} \le n^{O(1)} 
		e^{O(1/n)\,\xvec\trans\! A\xvec}$ 
		uniformly for $\xvec\in\Reals^n$.
	\end{itemize}
	Let $\Xvec$ be a Gaussian random vector 
	with density
	$\pi^{-n/2} \abs{A}^{1/2} \, e^{-\xvec\trans\!A\xvec}$.
	Then, provided $\V f(\Xvec)$ is finite
	and $h$ is bounded in~$\varOmega$,
	\[
	\int_\varOmega e^{-\xvec\trans\!A\xvec + f(\xvec)+h(\xvec)}\,d\xvec
	= (1+K) \pi^{n/2}\abs{A}^{-1/2} e^{\E f(\Xvec)+\frac12\V f(\Xvec)},
	\]     
	where, for sufficiently large $n$,
	\[
	    \abs{K} \le  e^{\frac12\Var\Im f(\Xvec)}\,
	    \( 3e^{\phi^3+e^{-\rho_1^2/2}}-3
	    +\sup_{\xvec\in\varOmega}|e^{h(\xvec)}-1|\).
	\]
\end{theorem}

\bigskip
Now we will prove Lemma~\ref{lem:inbox}.

\begin{proof}[Proof of Lemma~\ref{lem:inbox}]
Let $\rho=d^{-1/2}\log n$.
Applying Lemma~\ref{lem:truncate} gives
	\[
	\int_{\calR(\rho)} F(\thetavec) \, d\thetavec
	= \int_{\calR(\rho)} \exp\(- \thetavec\trans\! A \thetavec + f(\thetavec) + h(\thetavec)\)\,d\thetavec,
	\]
	where 
	\begin{align}
	f(\thetavec)&=\sum_{W\in\rsets} \sum_{p=3}^4 i^p c_p(\lambda_W)\, \Bigl(\, \sum_{j\in W} \theta_j\Bigr)^{p}, \nonumber\\ 
	h(\thetavec) &= O(nr^4d^{-3/2}\log^5 n) \stackrel{\eqref{eq:assumptions}}{=} O(nr^2/d).\label{eq:h-theta}
	\end{align}
	
	We will apply Theorem~\ref{gauss4pt}  with $\Omega = \calR(\rho)$. 
Let $T,\rho_1,\rho_2$ be as in Lemma~\ref{lem:diaggeneral-r}. 
Then $T(U_n(\rho_1))\subseteq \mathcal{R}(\rho)\subseteq T(U_n(\rho_2))$.
	Observe that $\rho_1,\rho_2=\Theta(\rho d^{1/2})=\Theta(\log{n})$, by \eqref{eq:dlacomp}.
	Clearly $\rho_1\le \rho_2$ and thus condition (a) in Theorem~\ref{gauss4pt} is satisfied. 
	
Now for $j\in [n]$,
	\begin{align*}
		\frac{\partial f}{\partial \theta_j}(\thetavec)&=\dfrac{1}{6}\sum_{W\ni j} \lambda_{W}(1-\lambda_W)(1-6\lambda_{W}+6\lambda_W^2)\biggl(\,\sum_{\ell\in W}\theta_\ell\biggr)^{\!3}\\
		&\hspace*{3cm} {} - \dfrac{i}{2}\sum_{W\ni j} \lambda_{W}(1-\lambda_W)(1-2\lambda_{W})\biggl(\,\sum_{\ell\in W}\theta_\ell\biggr)^{\!2}.
	\end{align*}
	
Thus, for all $\thetavec\in T(U_n(\rho_2))$ and all $j\in [n]$ we have
	\begin{equation}\label{eq:partialderivg}
		\left|\frac{\partial f}{\partial \theta_j}(\thetavec)\right|
=O\left(\varLambda \binom{n-1}{r-1} \,r^2\, \norm{\thetavec}_{\infty}^2\right)
=O\left(\varLambda \binom{n-1}{r-1} \,r^2\, \rho^2\right),
	\end{equation}
by Lemmas~\ref{lem:lambdaW-ratios} and \ref{lem:dlacomp}
and using the fact that $r\rho = o(1)$. 
Hence, by \eqref{eq:partialderivg} and Lemma~\ref{lem:diaggeneral-r},
	\begin{equation}\label{eq:phi1bound1}
	2\rho_2\, \norm{T}_1\, \left|\frac{\partial f}{\partial \theta_j}(\thetavec)\right|\stackrel{}{=}
	O\biggl(\log{n}\cdot \varLambda^{-1/2}\binom{n-1}{r-1}^{\!-1/2}
	 \!\!\varLambda \binom{n-1}{r-1} r^2\, \rho^2\biggr)
	    \stackrel{\eqref{eq:dlacomp}}{=}O\left(\frac{r^2 \log^3{n} }{d^{1/2}}\right)
	\end{equation}
for every $\thetavec\in T(U_n(\rho_2))$ and $j\in [n]$. 
	Also for all $j,k\in [n]$ (including $j=k$),
	\begin{align*}
		\frac{\partial^2 f}{\partial \theta_j\, \partial \theta_k}(\thetavec)
		 &=\dfrac{1}{2}\sum_{W\supset \{j,k\}} \lambda_{W}(1-\lambda_W)(1-6\lambda_{W}+6\lambda_W^2)\biggl(\,\sum_{\ell \in W}\theta_\ell \biggr)^{\!2}\\
		&{\qquad} - i\sum_{W\supset \{j,k\}} \lambda_{W}(1-\lambda_W)(1-2\lambda_{W})\biggl(\,\sum_{\ell\in W}\theta_\ell\biggr).
	\end{align*}
Arguing as above, if $\thetavec\in T(U_n(\rho_2))$ then 
	\begin{equation}\label{eq:partialderiv2g}
		\left|\frac{\partial^2 f}{\partial \theta_j\, \partial \theta_k}(\thetavec)\right|= 
		\begin{cases}
			\,\displaystyle O\left(\varLambda \binom{n-1}{r-1} \, r\, \norm{\thetavec}_{\infty}\right),  & \mbox{if } j=k; \\[3ex]
			\,\displaystyle O\left(\varLambda \binom{n-2}{r-2} \, r\, \norm{\thetavec}_{\infty}\right), & \mbox{otherwise.}
		\end{cases}
	\end{equation}
	Then \eqref{eq:partialderiv2g} and Lemma~\ref{lem:diaggeneral-r} imply that
	\begin{align}\label{eq:phi1bound2}
	4\rho_2^2\, \norm{T}_1\, &\norm{T}_\infty \, \norm{H}_{\infty} \notag \\
	&=O\biggl(\log^2 n\frac{1}{\varLambda \binom{n-1}{r-1}}\varLambda\left(\binom{n-1}{r-1}+ (n-1)\binom{n-2}{r-2}\right)r \rho\biggr)=O\left(\frac{r^2 \log^3 n}{d^{1/2}}\right).
	\end{align}
	By \eqref{eq:phi1bound1} and \eqref{eq:phi1bound2} there exists
	\begin{equation}
\label{eq:phi1}
\phi=O\left(\frac{r^2\, n^{1/3} \log^3{n} }{d^{1/2}}\right)
\end{equation}
such that the left side of both \eqref{eq:phi1bound1} and \eqref{eq:phi1bound2} are at most $\phi n^{-1/3}$.
	
Recall that the 2-norm
of the real symmetric matrix $A^{-1}$ equals the largest eigenvalue of~$A^{-1}$.
Using this we obtain
\begin{align*}
f(\thetavec) &= O\((r\norm{\thetavec}_\infty + r^2\norm{\thetavec}_\infty^2)\, \thetavec\trans\! A\thetavec\)
 = O\((1 + r^2\norm{\thetavec}_2^2)\, \thetavec\trans\! A\thetavec\) \\
 &= O\(\thetavec\trans\! A\thetavec +
   n^2(\thetavec\trans\! A\thetavec)^2 \norm{A^{-1}}_2\) \\
 &\stackrel{\eqref{sigmabound}}{=}
   O\biggl(\thetavec\trans\! A\thetavec +
      \frac{n^2(\thetavec\trans\! A\thetavec)^2} 
      {\varLambda\binom{n-1}{r-1}}\biggr)
 \stackrel{\eqref{mainineq}}{=}
 O\(\thetavec\trans\! A\thetavec +
       n(\thetavec\trans\! A\thetavec)^2\) \\
  & = O\( n^3 e^{\thetavec\trans\! A\thetavec/n}\),
\end{align*}
so condition (c) is satisfied. 

	By Theorem~\ref{gauss4pt} we have
	\[\int_{U_n(\rho)} F(\thetavec) \, d\thetavec= (1+K)\frac{\pi^{n/2}}{|A|^{1/2}}\,\,
	    \exp\left(\E f(\Xvec)+\dfrac{1}{2}\V f(\Xvec)\right), 
	\]
	where 
	\[ K \le e^{\Var(\Im f(\Xvec))/2}\, \Bigl( O\(\dfrac{nr^2}{d}\)+3e^{\phi^3+e^{-\rho_1^2/2}}-3\Bigr)
	=O\Bigl(\dfrac{nr^2}{d}+\phi^3+e^{-\rho_1^2/2}\Bigr)\, e^{\Var(\Im f(\Xvec))/2}.
	\]
In the last step we use the fact that $\phi=o(1)$ and $e^{-\rho_1^2/2}=o(1)$.
  The $nr^2/d$ term inside the $O(\cdot)$ is the bound on $h$ from (\ref{eq:h-theta}).
	To complete an estimate of $K$, it remains to bound 
	\[
	\Var\(\Im f(\Xvec)\)=\Var\biggl(\dfrac{1}{6}\sum_{W\in\rsets} \lambda_W(1-\lambda_W)(1-2\lambda_W)\, \Bigl(\, \sum_{j\in W} X_j\Bigr)^{\!3}\,\biggr).
	\]
	We will rely heavily on Isserlis' theorem (also called Wick's formula)
	in order to establish bounds for the variance of $\Im f(\Xvec)$ and later for the pseudovariance of $f(\Xvec)$.
	Isserlis' theorem states that the expected value
	of a product of jointly Gaussian random variables, each with zero mean, can be
	obtained by summing over all partitions of the variables into pairs, where the
	term corresponding to a partition is just the product of the covariances of each pair.
	See for example~\cite[Theorem~1.1]{MNBO}.

	In particular, for a normally distributed random vector $(Y_1,Y_2)$
	  with expected value $(0,0)$, we have
	\begin{align*}
		\E (Y_1^3)&=0, \hspace*{2cm} \E (Y_1^4)=3\Cov(Y_1,Y_1),\\
		\E(Y_1^3\, Y_2^3)&=9 \Cov(Y_1,Y_1)\, \Cov(Y_2,Y_2)\, \Cov(Y_1,Y_2)+6\Cov(Y_1,Y_2)^3,\\
		 \E (Y_1^4\, Y_2^4)&=9\Cov(Y_1,Y_1)^2 \Cov(Y_2,Y_2)^2+72\Cov(Y_1,Y_1)\Cov(Y_2,Y_2)\Cov(Y_1,Y_2)^2 \\ & \qquad {} +24\Cov(Y_1,Y_2)^4.
	\end{align*}
	Since the sum of components of a normally distributed random vector is also normally distributed,
	 we can apply Isserlis' theorem to sums involving the random variables $X_j$, $j\in [n]$.
	Then for any $W\in \rsets$ we have

	\begin{equation}\label{eq:cubzero}
		\E \biggl[\Bigl(\, \sum_{j\in W} X_j\Bigr)^{\!3}\,\biggr]=0,
	\end{equation}
        and so
	\[
	\Var\(\Im f(\Xvec)\)=\sum_{W\in\rsets}\,\sum_{W'\in\rsets}
	   \!O(\varLambda^2)\; \E \biggl[\Bigl(\, \sum_{j\in W} X_j\Bigr)^{3}\Bigl(\, \sum_{k\in W'} X_k\Bigr)^{3}\biggr].
	\]
	For $W,W'\in\rsets$ let 
	\[
	\sigma(W,W'):=\Cov\biggl[\,\sum_{j\in W} X_j,\sum_{k\in W'} X_k\biggr].
	\]
	Now $\Cov[X_j,X_k]$ equals the corresponding values of $(2A)^{-1}$ and hence,
	by Lemma~\ref{lem:diaggeneral-r} and (\ref{eq:dlacomp}),
	\[
	\Cov\left[X_j,X_k\right]=
	\begin{cases}
		O\(\dfrac{1}{d}\),  & \mbox{if } j=k; \\[1ex]
		O\(\dfrac{1}{nd} \), & \mbox{otherwise.}
	\end{cases}
	\]
	Since covariance is additive, we have
	\begin{equation}\label{eq:covg}
           \sigma(W,W') = O\biggl(\frac{r^2}{nd} + \frac{\card{W\cap W'}}{d}\biggr). 
	\end{equation}
Using this together with Isserlis' theorem, for any pair $W,W'$,
	\begin{align*}
	\E \biggl[\Bigl(\, \sum_{j\in W} X_j\Bigr)^{3}\Bigl(\, \sum_{k\in W'} X_k\Bigr)^{3}\,\biggr] 
	 &= 9\, \sigma(W,W)\, \sigma(W',W')\, \sigma(W,W') +6 \, \sigma(W,W')^3 \\[-1ex]
         &= O\left(\frac{r^2}{d^2}\, \sigma(W,W')\right)\\
	 &= O\biggl(\frac{r^4}{nd^3} + \frac{r^2\,\card{W\cap W'}}{d^3}\biggr).
	\end{align*}
	The average value of $\card{W\cap W'}$ over pairs of $r$-sets is $r^2/n$,
	so we can sum over $W,W'\in\rsets$ to obtain
	\[
	    \Var(\Im f(\Xvec)) = O\biggl(\varLambda^2\, \binom{n}{r}^{\!2} 
	         \biggl(\frac{r^4}{nd^3} + \frac{r^2\,(r^2/n)}{d^3}\biggr)\biggr)
	         \stackrel{(\ref{eq:dlacomp})}{=} O\biggl( \frac{nr^2}{d} \biggr).
	\]
	By \eqref{eq:assumptions} this term tends to $0$, implying that 
	$K=O(nr^2/d + \phi^3+e^{-\rho_1^2})$.

	All that is left is to establish bounds on $\E f(\Xvec)$ and $\V f(\Xvec)$.
	Due to~\eqref{eq:cubzero}, we have
	\begin{align*}
	\E f(\Xvec)&=\dfrac{1}{24}\sum_{W\in\rsets}\lambda_W(1-\lambda_W)(1-6\lambda_W + 6\lambda_W^2)\,
	\E \biggl[\Bigl(\, \sum_{j\in W} X_j\Bigr)^{\!4}\,\biggr]\\
  &=O\biggl(\varLambda \sum_{W\in\rsets} \E \biggl[\Bigl(\, \sum_{j\in W} X_j\Bigr)^{\!4}\,\biggr]\biggr).
	\end{align*}
	Again using Isserlis' theorem, for any $W\in \rsets$ we have
	\[
	\E \biggl[\Bigl(\, \sum_{j\in W} X_j\Bigr)^{\!4}\,\biggr]=3\sigma(W,W)^2\stackrel{\eqref{eq:covg}}{=}O\left(\frac{r^2}{d^2}\right).
	\]
	Hence by (\ref{eq:dlacomp}),
	\[
	\E f(\Xvec)=O\biggl(\frac{nr^2}{d}\biggr).
	\]
	Now $\V f(\Xvec)$ satisfies
	\[
	  \abs{\V f(\Xvec)} = \abs{\E\, (f(\Xvec)-\E f(\Xvec))^2} \le \E\,\abs{f(\Xvec)-\E f(\Xvec)}^2
	=\Var(\Re f(\Xvec))+\Var(\Im f(\Xvec)).
	\]
	Since we already established a bound on $\Var(\Im f(\Xvec))$, we only need to consider $\Var(\Re f(\Xvec))$. Note that 
	\begin{align*} 
	\Var(\Re f(\Xvec))
 &\leq 
		\sum_{W\in\rsets}\,\sum_{W'\in\rsets}c_4(\lambda_W)c_4(\lambda_{W'})\,\,
\E \biggl[\Bigl(\, \sum_{j\in W} X_j\Bigr)^{\!4}\Bigl(\, \sum_{k\in W'} X_k\Bigr)^{\!4}\,\biggr]. 
\end{align*}
By Isserlis' theorem, we have
	\begin{align*}
&	\E \biggl[\Bigl(\, \sum_{j\in W} X_j\Bigr)^{\!4}\Bigl(\, \sum_{k\in W'} X_k\Bigr)^{\!4}\,\biggr] \\
	&\hspace*{2cm} {} = 9\sigma(W,W)^2 \sigma(W',W')^2+72\sigma(W,W)\sigma(W',W')\sigma(W,W')^2
		+24\sigma(W,W')^4.
	\end{align*}
     Since $\sigma(W,W')=O(r/d)$ from~\eqref{eq:covg}, 
     \[
          \Var(\Re(f(\Xvec))) = O\biggl( \varLambda^2 \binom{n}{r}^2\, \frac{r^4}{d^4}\biggr) 
 \stackrel{(\ref{eq:dlacomp})}{=}
O\biggl( \frac {n^2 r^2}{d^2} \biggr) \stackrel{(\ref{eq:assumptions})}{=} 
O\biggl( \frac {n r^2}{d} \biggr).
    \]
	Therefore $|\V(f(\Xvec))| = O(nr^2/d)$ and hence
	\begin{align*}
	\int_{\calR(\rho)} F(\thetavec) \, d\thetavec 
	&= \frac{\pi^{n/2}}{|A|^{1/2}}\, \exp\Bigl(\E(f(\Xvec))+ \dfrac{1}{2}\V{f(\Xvec)}+O\Bigl(\dfrac{nr^2}{d}+\phi^3+e^{-\rho_1^2}\Bigr) \Bigr)\\
	&=\left(1+O\left(\frac{nr^2}{d}+\frac{r^6 n \log^9{n}}{d^{3/2}}+n^{-{\Omega(\log{n})}}\right) \right)\frac{\pi^{n/2}}{|A|^{1/2}},
	\end{align*}
using (\ref{eq:phi1}) and the definition of $\rho_1$.
\end{proof}

\subsection{Proof of Lemma~\ref{L:step1}}

In this section we evaluate the integral over the region $U_n(r^{-1})\setminus \calR(\rho)$.
The following technical bound will be useful: 
 for any $t \in \Reals$ and $\lambda \in [0,1]$, we have
 \begin{equation}\label{factor-bound}
 	|1+  \lambda (e^{it}-1)| \leq \exp \left(-\dfrac12\left(1-  \dfrac{t^2}{12}\right) \lambda(1-\lambda) t^2 \right).
 \end{equation}

 \begin{proof}[Proof of Lemma~\ref{L:step1}]
We will show that for  any $\hat{\rho}$ satisfying
$(2r)^{-1} \geq \hat{\rho} \geq d^{-1/2}\log{n}$, we have
\begin{equation}\label{eq:unionassumption}
\int_{U_n(2 \hat{\rho}) \setminus \calR(\hat{\rho})} |F(\thetavec)| \, d \thetavec
= n^{-\Omega(\log n)}\frac{\pi^{n/2}  }{|A|^{1/2}}.
\end{equation}
Observe that 
\[  U_n((2r)^{-1})\setminus \calR(d^{-1/2}\log n) \,\subseteq\, \bigcup_{\ell=0}^{L-1} \,
\(U_n(2^{\ell+1}d^{-1/2}\log n)\setminus\calR(2^\ell d^{-1/2}\log n)\) \]
for $L = \lceil\log_2 \( (2r)^{-1}/(d^{-1/2}\log n)\) \rceil  = O(r \log n)$,
and that 
\[ U_n(r^{-1}) \setminus \calR(d^{-1/2}\log n) \,\subseteq\, 
  \left( U_n(r^{-1}) \setminus \calR((2r)^{-1})\right) \cup
  \left( U_n((2r)^{-1})\setminus \calR(d^{-1/2}\log n) \right).\]
This expresses the region of integration in the lemma statement as a union of integrals of the form given in
(\ref{eq:unionassumption}), and the result follows.

It remains to prove \eqref{eq:unionassumption}.
 Using \eqref{factor-bound},  for  any such 
 $\hat{\rho}$ 
 \[
        \int_{U_n(2 \hat{\rho}) \setminus \calR(\hat{\rho})} |F(\thetavec)| \, d \thetavec
        \leq \int_{\Reals^n \setminus \calR(\hat{\rho})}   e^{- (1 - r^2 \hat{\rho}^2/3)\, \thetavec\trans\!  A \thetavec} \, d \thetavec.
\]
Let $T$ be as in Lemma~\ref{lem:diaggeneral-r} and note that $|T|=|A|^{-1/2}$.
Then by Lemma~\ref{lem:diaggeneral-r} and \eqref{eq:dlacomp} there exists a 
$\hat{\rho}_1=\Theta(\hat{\rho} d^{1/2})$ such that $T(U_n(\hat{\rho}_1))\subseteq \calR(\hat{\rho})$.
Taking $\hat{\rho}_1' = (1-r^2 \hat{\rho}^2/3)^{1/2}\, \hat{\rho}_1$ we find that
$(1-r^2 \hat{\rho}^2/3)^{-1/2}\, U_n(\hat{\rho}_1') = U_n(\hat{\rho}_1)$ and hence
\[ (1-r^2 \hat{\rho}^2/3)^{-1/2}\, T(U_n(\hat{\rho}_1')) = T(U_n(\hat{\rho}_1))\subseteq \calR(\hat{\rho}).\]
Therefore, substituting $\thetavec=(1-r^2 \hat{\rho}^2/3)^{-1/2}\, T\xvec$ gives
\begin{equation*}
   \int_{\Reals^n \setminus \calR(\hat{\rho})}   e^{- \left(1 - r^2  \hat{\rho}^2/3  \right)
    \thetavec\trans\!  A \thetavec} \, d \thetavec \leq  \frac{(1 - r^2 \hat{\rho}^2/3 )^{-n/2}}{|A|^{1/2}} \,
     \int_{\Reals^n \setminus  U_n(\hat{\rho}_1') } e^{-\xvec\trans \xvec}\,  d \xvec.
\end{equation*}     
Note that $(1 - r^2 \hat{\rho}^2/3 )^{-n/2}= \exp(O(r^2 \hat{\rho}^2 n))$. In addition we have $\hat{\rho}_1'=\Theta(\hat{\rho}_1)=\Theta(\hat{\rho}d^{1/2})$ and thus 
\[
\int_{\Reals^n \setminus  U_n(\hat{\rho}_1') } e^{-\xvec\trans \xvec}\,  d \xvec \le n \exp(-\Omega(\hat{\rho}_1^2))=n \exp(-\Omega(\hat{\rho}^2 d)).
\]
We deduce that
\begin{equation*}     
  \int_{\Reals^n \setminus \calR(\hat{\rho})}   e^{- (1 - r^2  \hat{\rho}^2/3  )
   \thetavec\trans\!  A \thetavec} \, d \thetavec   \leq n\, \exp
  \(O(r^2\hat{\rho}^2n)-  \Omega(\hat{\rho}^2 d)\)\, \frac{1}{|A|^{1/2}}
        =  n^{-\Omega(\log n)}\frac{\pi^{n/2}  }{|A|^{1/2}},
\end{equation*}
    as $d\gg r^2 n$, by (\ref{eq:assumptions}), and $\hat{\rho}^2 d=\Omega(\log^2 {n})$.
  \end{proof}
  
\subsection{Proof of Lemma~\ref{lem:reduce-to-J0}}

In this section we complete the evaluation of the integral by examining the values in the region $U_n(\pi)\setminus \mathcal{B}$.
For $x\in\Reals$, define
$\abs{x}_{2\pi} = \min_{k\in\Integers}\, \abs{x-2k\pi}$ and
note that $\abs{1+\lambda(e^{ix}-1)}$ depends only on $\abs{x}_{2\pi}$.

 \begin{proof}[Proof of Lemma~\ref{lem:reduce-to-J0}]
Let $\thetavec\in U_n(\pi)\setminus \mathcal{B}$.
 	First suppose that $|\theta_a- \theta_b|_{2\pi}> (2r)^{-1}$ for some  $a,b\in [n]$.
 	For any $W_1, W_2 \in \rsets$ that $W_1 \mathbin{\triangle} W_2 =\{a,b\}$, we have 
 	\[
 		\left|\,\sum_{j\in W_1} \theta_j - \sum_{j\in W_2} \theta_j\right|_{2\pi}\!\!>(2r)^{-1}.
 	\]
 	So   $\Abs{\sum_{j\in W_1} \theta_j}_{2\pi} >(4r)^{-1}$  or
	 $\Abs{\sum_{j\in W_2} \theta_j}_{2\pi}>(4r)^{-1}$, or both. In any case, by Lemma~\ref{lem:lambdaW-ratios} and \eqref{factor-bound} we have
 	\begin{equation}\label{eq:W12}
 		\Abs{1 +  \lambda_{W_1}(e^{i\sum_{j\in W_1} \theta_j}-1)}\cdot 
 		\Abs{1 +  \lambda_{W_2}(e^{i\sum_{j\in W_2} \theta_j}-1)}
 		\leq e^{-\Omega(\varLambda/r^2)}.
 	\end{equation}
 	Note that there are exactly $\binom{n-2}{r-1}=\Theta\( \binom{n-1}{r-1}\)$ pairs
	$W_1,W_2$ such that $W_1 \mathbin{\triangle} W_2 =\{a,b\}$.
Furthermore, every $W\in\rsets$ is contained in at most one such pair.
Then, multiplying inequalities \eqref{eq:W12} for all such pairs, 
we obtain 
\[  \abs{F(\thetavec)} = \exp \biggl(-\Omega\biggl( \varLambda \binom{n-1}{r-1}/r^2\biggr)\biggr)
    \stackrel{\eqref{eq:dlacomp}}{=} e^{-\Omega(d/r^2)}.
\]
By~\eqref{eq:assumptions}, $\dfrac {d}{r^2}\gg nr^2\log n$, while by Lemma~\ref{lem:detA}
and because $d<n^r$, we have
$\abs{A} = e^{O(n\log d)} = e^{O(nr\log n)}$.
Therefore the total contribution to the integral from this case is at most
\[
      (2\pi)^n e^{-\Omega(d/r^2)} = e^{-\omega(nr^2\log n)} = 
        n^{-\omega(n)}\, \frac{\pi^{n/2}  }{|A|^{1/2}}.
 \]
 	
 	All remaining points $\thetavec\in U_n(\pi)\setminus \mathcal{B}$ satisfy
 	$|\theta_a -\theta_b|_{2\pi} \leq (2r)^{-1}$ for all $a,b\in [n]$ 
and $\min_{j \in [n],\, k\in [r]}|\theta_j - \frac{2\pi k}{r}|_{2\pi} > (2r)^{-1}$.
These two conditions imply that for any such $\thetavec$ there exists $k\in [r]$
such that for all $j\in [n]$ we have
\[ \frac{2\pi k}{r} + \frac{1}{2r} < \theta_j < \frac{2\pi(k+1)}{r} - \frac{1}{2r}.\]
Summing the above over any $W\in\rsets$ implies that
 	$\dfrac{1}{2} \leq \left|\sum_{j\in W} \theta_j\right|_{2\pi} \leq \pi$. 
Hence (\ref{factor-bound}) implies that
\[ |F(\thetavec)| = \exp\left( - \Omega(\varLambda) \, \binom{n}{r}\right).\]
Again, multiplying by $(2\pi)^n$ for an upper bound, we see that the contribution
of all such points $\thetavec$ to the integral is at most
 	\begin{align*}
 		 (2\pi)^n\, \exp\left( - \Omega(\varLambda) \, \binom{n}{r}\right)
   = \exp \left(-\Omega\left( \varLambda \binom{n-1}{r-1}\right)\right)
= n^{-\omega(n)}\, \frac{\pi^{n/2}  }{|A|^{1/2}},
 	\end{align*}
completing the proof.
 	\end{proof}	

\section{Solving the beta-system}\label{s:proof-degree-sufficient}

We first prove that the solution to (\ref{exact}) is unique if it exists.
\begin{proof}[Proof of Lemma~\ref{lem:unique}]
Suppose $\betavec'\ne \betavec''$ both satisfy~\eqref{exact}.
For $y\in\Reals$ and $W\in\rsets$ define 
$\xi_W(y):=(1-y)\lambda_W(\betavec') + y\lambda_W(\betavec'')$.
Consider the entropy function
\[   
   S(y):= \sum_{W\in\rsets} \Bigr(\xi_W(y)\log\frac 1{\xi_W(y)}
                                  + (1-\xi_W(y))\log\frac 1{1-\xi_W(y)}\Bigr).
\]
The derivative of $S(y)$ at $y=0$ is
\begin{align*}
   S'(0) &= \sum_{W\in\rsets}  \(\lambda_W(\betavec')-\lambda_W(\betavec'')\)
      \log\frac{\lambda_W(\betavec')}{1-\lambda_W(\betavec')} \\
  &{} \stackrel{\eqref{eq:lambdaW-def}}{=}
   \sum_{W\in\rsets} \(\lambda_W(\betavec')-\lambda_W(\betavec'')\)
      \sum_{j\in W} \beta'_j \\
  &{}= \sum_{j=1}^n \beta'_j\,\sum_{W\ni j} \(\lambda_W(\betavec')-\lambda_W(\betavec'')\)
  \stackrel{\eqref{exact}}{=} 0.
\end{align*}
Similarly, the derivative of $S(y)$ at $y=1$ is  $S'(1)=0$.

On the other hand, $\betavec'\ne\betavec''$ implies that
   $\lambda_W(\betavec')\ne\lambda_W(\betavec'')$ for at least one
$W\in\rsets$.
The second derivative of $S(y)$ equals
\[ - \sum_{W\in\rsets} \bigl(\lambda_W(\betavec'')-\lambda_W(\betavec')\bigr)^2\,
  \xi_W(y)^{-1}\, (1- \xi_W(y))^{-1},\]
and hence is strictly negative when $\betavec'\neq \betavec''$.
Therefore $S(y)$ is strictly concave
and cannot have more than one stationary point. This completes the proof.
\end{proof}

To prove  Lemma~\ref{lem:degree-sufficient} we will  employ the following lemma from \cite{sandwichGIM}.
\begin{lemma}\label{l:Kowa}\emph{\cite[Lemma 7.8]{sandwichGIM}}\\
	Let $\Psi : \Reals^n \to \Reals^n$, $\eta>0$,
	and   $U  =  \{\betavec\in \Reals^n \st \|\betavec - \betavec^{(0)} \| \leq \eta \norm{\Psi(\betavec^{(0)})} \}$
	and $\betavec^{(0)}\in \Reals^n$,  where 
	$\|\cdot\|$ is  any vector norm in $\Reals^n$.
	Assume that
	\[
	\text{$\Psi$ is analytic in $U$} \qquad \text{and} \qquad \sup_{\xvec \in U} \|J^{-1}(\betavec)\| < \eta, 
	\]
	where $J$ denotes the Jacobian matrix   of $\Psi$ and  $\|\cdot\|$ stands for the induced matrix norm.
	Then there exists $\betavec^*\in U$ such that 
	$\Psi(\betavec^*) = \zerovec$. 
\end{lemma}

In connection with the system of  \eqref{exact}, we consider
$\Psi:\Reals^n\to \Reals^n$ defined by
\begin{equation}
\label{eq:Psi-def}
\Psi_j(\betavec)= \sum_{W \ni j} \lambda_W(\betavec)-d_j.
\end{equation}
Clearly, $\Psi$ is analytic in $\Reals^n$. Observe that 
\[ \frac{d}{dx} \left(\frac{e^{x+X}}{1+e^{x+X}}\right)=\frac{e^{x+X}}{1+e^{x+X}}\left(1-\frac{e^{x+X}}{1+e^{x+X}}\right)\]
and thus $ J(\betavec) = 2A(\betavec) $, where $J(\betavec)$ is the Jacobian matrix of $\Psi(\betavec)$ and $A(\betavec)$ is defined by \eqref{eq:A-def}.
We start by bounding  $ \|J^{-1}(\betavec)\|_\infty $ as required  for Lemma~\ref{l:Kowa}. 

\begin{lemma}\label{lem:invjacobian}
Let $\betavec^{(0)}\in \Reals^n$ and real numbers $\delta_1,\delta_2\ge 0$ satisfy $\max_{j,k\in[n]}|{\beta}^{(0)}_j-{\beta}^{(0)}_k|\le \delta_1/r$ and $e^{\delta_2}\lambda(\betavec^{(0)})\le 7/8$. 
Suppose that $n\ge 16e^{4\delta_1+8\delta_2}$. 
Then for any $\betavec\in \Reals^n$ such that $\norm{\betavec-\betavec^{(0)}}_{\infty}\le \delta_2/r$, we have 
\[\norm{J^{-1}(\betavec)}_{\infty} = \norm{(2A(\betavec))^{-1}}_{\infty}\le 2^8 C \frac{e^{36\delta_1+73\delta_2}}{\binom{n-1}{r-1}\lambda(\betavec^{(0)})},\]
where $C$ is the constant from Lemma~\emph{\ref{lem:diaggeneral-r}}.
\end{lemma}

\begin{proof}
Let $\betavec\in \Reals^n$ satisfy
$\norm{\betavec-\betavec^{(0)}}_{\infty}\le \delta_2/r$. Then
\[\max_{j,k\in[n]}|{\beta}_j-{\beta}_k|\le \max_{j,k\in[n]}|{\beta}^{(0)}_j-{\beta}^{(0)}_k|+
2\norm{\betavec - \betavec^{(0)}}_{\infty}\le \frac{\delta_1+2\delta_2}{r}.\]
Applying Lemma~\ref{lem:diaggeneral-r} for $\betavec$ implies for all sufficiently large $n$ that
\[\norm{(2A(\betavec))^{-1}}_{\infty}\le C \frac{e^{35\delta_1+70\delta_2}}{\varLambda(\betavec) \binom{n-1}{r-1}}.\]
By Lemma~\ref{lem:lambdaW-ratios-different-beta} and our assumptions we have $\lambda(\betavec)\le e^{\delta_2}\lambda(\betavec^{(0)})\le 7/8$.
Therefore the conditions of Lemma~\ref{lem:dlacomp} are satisfied and we have
\[\norm{(2A(\betavec))^{-1}}_{\infty}\le 2^8 C \frac{e^{36\delta_1+72\delta_2}}{\binom{n-1}{r-1}\lambda(\betavec)}.\]
The result follows as $\lambda(\betavec)\ge e^{-\delta_2}\lambda(\betavec^{(0)})$ by Lemma~\ref{lem:lambdaW-ratios-different-beta}.
\end{proof}

Further, we explain how to carefully choose $U$ and $\betavec^{(0)}$   depending on whether  $d$ is small relative to $\binom{n-1}{r-1}$ or not. 

\subsection{Proof of Lemma~\ref{lem:degree-sufficient}(i)}\label{ss:unique}

Recalling (\ref{eq:density}), define
	\[ \betavec^{(0)}:= \biggl( \frac 1r\log\frac{\lambda}{1-\lambda},
	\ldots, \frac 1r\log\frac{\lambda}{1-\lambda}\biggr) \]
and note that $\norm{\Psi(\betavec^{(0)})}_\infty= \max_{j\in [n]}|d-d_j|$.
Define 
\[U:=\bigl\{\betavec:\norm{\betavec-\betavec^{(0)}}_{\infty}\le \eta \norm{\Psi(\betavec^{(0)})}_\infty\bigr\}
=\bigl\{\betavec:\norm{\betavec-\betavec^{(0)}}_{\infty}\le  \eta\, \max_{j\in[n]}|d-d_j| \bigr\},\]
where $\eta=2^{10} C/d$ and $C$ is the constant from Lemma~\ref{lem:invjacobian}. 
Since $\max_{j,k\in[n]}|{\beta}^{(0)}_j-{\beta}^{(0)}_k|=\nobreak0$, we set $\delta_1:= 0$.
Now assume that $\varDelta$ is sufficiently small, in particular 
$\varDelta\le \varDelta_0 := \min\{ ( 2^{17} C)^{-1},1\}$. Then for any $\betavec\in U$,
\begin{equation}
\label{eq:64r}
\norm{\betavec-\betavec^{(0)}}_{\infty}\le \eta d\, \( e^{\varDelta/r} - 1\) 
          \leq 2\eta d\varDelta/r = \frac{2^{11} C}{d} \cdot \frac{d \varDelta }{r}\le \frac{1}{64 r}.
\end{equation}
Hence we define $\delta_2:=1/64$.
Since
\[\lambda(\betavec^{(0)})=d\binom{n-1}{r-1}^{-1}\stackrel{\eqref{eq:assumptions}}{\le} \dfrac{1}{2},\]
we deduce that
\[\lambda(\betavec^{(0)})\, e^{\delta_2}\le e^{1/64} \lambda(\betavec^{(0)})\le e^{1/64}/2\le \dfrac78. \]
Therefore the conditions of Lemma~\ref{lem:invjacobian} are met for $\delta_1$ and $\delta_2$ as
above, and we deduce for every $\betavec\in U$,
\[ \norm{J^{-1}(\betavec)}_{\infty} = \norm{(2A(\betavec))^{-1}}_{\infty}\le 2^8 C \frac{e^{73\delta_2 }}{\lambda(\betavec^{(0)})\binom{n-1}{r-1}}< \frac{2^{10}C}{d}=\eta.\]
Hence all the conditions of Lemma~\ref{l:Kowa} hold, and applying this lemma 
shows that there exists a solution $\betavecstar$ to (\ref{exact}).
Finally note that (\ref{eq:64r}) implies that 
$\max_{j,k\in [n]} |\beta^\ast_j-\beta^\ast_k| = O(1/r)$, completing the proof.

\subsection{Proof of Lemma~\ref{lem:degree-sufficient}(ii)}

For part (ii), we define $\betavec^{(0)} = (\beta_1^{(0)}, \ldots,\beta_n^{(0)})\trans $ by
\[
    \beta_j^{(0)} := \log d_j  - \frac{1}{r} \log S,
\]
where
\[
     S  :=  \frac{n-r+1}{n}  \sum_{W\in\mathcal{S}_{r-1}(n)}\, \prod_{k \in W} d_k.
\]
Note that $\max_{j,k\in[n]}|\beta^{(0)}_j - \beta^{(0)}_k| = \max_{j,k\in[n]}|\log d_j  - \log d_k|\leq   2 \varDelta/r$.
Define 
\[
    U :=  \bigl\{ \betavec \st \|\betavec - \betavec^{(0)}\|_\infty \leq \varDelta/r \bigr\}.
\]

For any $W \in \rsets$, using the assumptions of the lemma we have
\[
\lambda_W(\betavec^{(0)})= \frac{\exp\left(\sum_{k \in W} \beta^{(0)}_k \right)}{1+\exp\left(\sum_{k \in W} \beta^{(0)}_k \right)}  = O(1)\, \frac{\prod_{k\in W} d_k}{ S }
   = O(1)\, \frac{d^r}{S}.
\]
Furthermore, 
\[ 
  S= \Omega\left(\frac{n-r+1}{n}\, \binom{n}{r-1} d^{r-1}\right) = \Omega\left(\binom{n-1}{r-1}\, d^{r-1}\right),
\]
and so, using our assumption on $rd$,
\[
 \lambda_W(\betavec^{(0)}) = O\left(\frac{d}{\binom{n-1}{r-1}}\right) = o(r^{-1}).
 \]
 
It follows that for all $j\in [n]$,
Lemma~\ref{lem:lambdaW-ratios} implies that 
$\lambda_W(\betavec^{(0)})=\Theta(\lambda(\betavec^{(0)}))$,
and hence
\begin{align*}
\lambda_W(\betavec^{(0)})= \frac{\exp\left(\sum_{k \in W} \beta^{(0)}_k \right)}{1+\exp\left(\sum_{k \in W} \beta^{(0)}_k \right)}  
&= \(1+ O(\lambda(\betavec^{(0)}))\) \, \frac{\prod_{k\in W} d_k}{ S }\\
&= \(1+ o(r^{-1})\) \, \frac{\prod_{k\in W} d_k}{ S }.
\end{align*}
It follows that for all $j\in [n]$,
\[
\sum_{W \ni j} \lambda_{W} (\betavec^{(0)}) = d_j\, \(1+ o(r^{-1})\) \,
\, \frac{\sum_{W\ni j}   \prod_{k\in W-j} d_k  }{S}.
\]
Next, we observe that the quantity 
$\sum_{W \ni j } \prod_{k\in W-j} d_k $  depends insignificantly on $j$. 
Indeed,  by our assumptions   we have
\[ \sum_{W \ni \ell }\, \prod_{k\in W-\ell} d_k  = \Theta(1)\, \binom{n-1}{r-1}d^{r-1} 
\]
for $\ell\in \{j,j'\}$,
while
\begin{align*}
	\sum_{W \ni j }\, \prod_{k\in W-j} d_k - 
	\sum_{W \ni j' }\, \prod_{k\in W-j'} d_k 
	&= \sum_{\substack{W\in\mathcal{S}_{r-2}(n)\\j,j'\notin W}}   
	   (d_{j'} - d_{j}) \prod_{k\in W} d_k 
	\\ &\leq \binom{n-2}{r-2} d \(e^{\varDelta/r} -e^{-\varDelta/r}\)\, d^{r-2} e^{O(1)} \\ &= 
	O( n^{-1})  \binom{n-1}{r-1}d^{r-1}.
\end{align*}
The last line uses the fact that for any $x\in\Reals$ we have
	\begin{equation}\label{eq:exp-r}
		e^{x/r}-1\le \frac{e^x}{r}.
	\end{equation}
This shows that for any $j,j'\in [n]$,
\[
\frac{\sum_{W \ni j } \prod_{k\in W-j} d_k}{\sum_{W \ni j' } \prod_{k\in W-j'} d_k} = 
1+  O( n^{-1}).
\] 
Observe also that 
\[
\dfrac{1}{n} \sum_{j \in [n]} \sum_{W \ni j  } \prod_{k\in W-j} d_k  =
\dfrac{n-r+1}{n} \sum_{W\in\mathcal{S}_{r-1}(n)} \, \prod_{k \in W} d_k = S.
\]
Combining the above and  using the assumptions, we conclude that for all $j\in [n]$,
\begin{equation}
\label{eq:approx-sol}
\sum_{W \ni j } \lambda_{W} (\betavec^{(0)}) = 
\(1+ o(r^{-1}) + O( n^{-1})\) d_j  = (1+o( r^{-1})) d_j.
\end{equation}
Taking the average of (\ref{eq:approx-sol}) implies that
\[
\lambda(\betavec^{(0)})\binom{n-1}{r-1}=\Theta(d)
\quad \mbox{and} \quad 
\lambda(\betavec^{(0)})\, e^{\varDelta}=o(1).
\]
Applying Lemma~\ref{lem:invjacobian} with $\delta_1:=2\varDelta$ and $\delta_2:=\varDelta$,
 we conclude that for every $\betavec\in U$,
\[
\|J^{-1}(\betavec)\|_{\infty} 
=  \| (2A(\betavec))^{-1} \|_{\infty} = O \left(d^{-1}\right).
\]
By the definition of $\Psi$ and our assumptions on $d_j$, it follows from (\ref{eq:approx-sol}) that  $\|\Psi(\betavec^{(0)})\|_\infty = o(d/ r)$. 
Hence we can apply Lemma~\ref{l:Kowa} with 
$\eta := \varDelta (r \|\Psi(\betavec^{(0)}\|_\infty)^{-1} = \omega(d^{-1})$,
completing the proof.

\section{The near-regular case}\label{s:nearreg}

In this section we will prove Theorem~\ref{thm:nearreg}. 
As mentioned at the end of Section~\ref{s:intro}, we have omitted some of the calculations in this and the following section. These calculations can be verified using the identities in Section~\ref{appendix}.
It will be convenient for us to begin the analysis in the first quadrant.
By assumption~\eqref{nearreg}, Lemma~\ref{lem:degree-sufficient}(i)
guarantees the existence of a solution $\betavecstar=(\beta_1^\ast,\ldots,\beta_n^\ast)$
which satisfies~\eqref{beta-range}, 
and by Lemma~\ref{lem:unique} this solution is unique. Therefore we are justified
in applying Theorem~\ref{thm:main}.

Next, recalling (\ref{eq:density}),
define $\gammavec^\ast=(\gamma^\ast_1,\ldots,\gamma^\ast_n)$ by
\[
   \beta_j^\ast = \frac 1r \log \frac{\lambda}{1-\lambda}
   + \gamma^\ast_j, \qquad\text{for $j\in[n]$.}
\]
In the regular case, $\betavecstar$ satisfies~\eqref{exact} when
$\gammavec^\ast=\zerovec$.
For $W\in\rsets$, define $\gamma_W^\ast:= \sum_{j\in W} \gamma_j^\ast$.
In addition, for $W\in\rsets$ and $s\in\Naturals$, define
$\W_s=\W_s(W):=\sum_{j\in W} \delta_j^s$.

\begin{lemma}\label{betavalue}
  Under assumptions~\eqref{mainineq} and~\eqref{nearreg}
  in the first quadrant,
  there is a solution of~\eqref{exact} with
  \[
      \gamma^\ast_j = \frac{(n-1)\,\delta_j}{(1-\lambda)(n-r)d} 
                          - \frac{(n-2\lambda n-2r)n\,\delta_j^2}{2(1-\lambda)^2(n-r)^2d^2} 
                  +\frac{\delta_j^3}{3d^3} - \frac{rR_2}{2(n-r)^2 d^2}
                    + O(r^{-1}n^{-1}d^{-3/5})
  \]
  uniformly for $j\in [n]$.
\end{lemma}
\begin{proof}
  Equations~\eqref{exact} can be written as $\varPhi(\gammavec)=\deltavec$,
  where $\varPhi:\Reals^n\to\Reals^n$ is defined by
  \[
      \varPhi_j(\gammavec):= \lambda(1-\lambda)
         \sum_{W\ni j} \frac{e^{\gamma_W}-1}{1+\lambda(e^{\gamma_W}-1)}
  \]
for $j\in [n]$.
 Consider $\bar\gammavec=(\bar\gamma_1,\ldots,\bar\gamma_n)$
 defined by
 \[
     \bar\gamma_j:= 
       \frac{(n-1)\,\delta_j}{(1-\lambda)(n-r)d} 
                          - \frac{(n-2\lambda n-2r)n\,\delta_j^2}{2(1-\lambda)^2(n-r)^2d^2} 
                  +\frac{\delta_j^3}{3d^3} - \frac{rR_2}{2(n-r)^2 d^2}
                  +\frac{R_2}{2n(n-r)d^2}.
 \]
 The function $L(x)=(e^x-1)/(1+\lambda(e^x-1))$ has bounded fifth derivative 
 for $\lambda\in[0,1]$, $x\in[-1,1]$, so by Taylor's theorem we have in that
 domain that
 \begin{equation}\label{Ltaylor}
     L(x) = x + \(\dfrac12-\lambda\)x^2 + \(\dfrac16-\lambda+\lambda^2\)x^3
        + \(\dfrac 1{24}-\dfrac 7{12}\lambda+\dfrac32\lambda^2-\lambda^3\)x^4
        + O(\abs{x}^5).
 \end{equation}
 For $W\in\rsets$, define $\bar\gamma_W:=\sum_{j\in W} \bar\gamma_j$.
Now
\[ \bar\gamma_W= O\biggl(d^{-1}\sum_{j\in W} \delta_j\biggr)= O(\deltamax r d^{-1}),\]
which implies that
 $(\bar\gamma_W)^5 = O(r^{-1}n^{-1}d^{-3/5})$.
 Therefore, from~\eqref{Ltaylor} we have
 \begin{equation}\label{LgammaW}
 \begin{aligned} 
   L(\bar\gamma_W) &= \frac{(n-1)\,\W_1}{(1-\lambda)(n-r)d}
      + \frac{(n^2-2\lambda n^2-2n+1)\,\W_1^2}{2(1-\lambda)^2(n-r)^2d^2}
      + \frac{(n-3)n^2\,\W_1^3}{6(n-r)^3d^3} \\
      &{\quad} +\frac{n^4\,\W_1^4}{24(n-r)^4d^4}
      - \frac{n(n-2\lambda n-2r)\,\W_2}{2(1-\lambda)^2(n-r)^2d^2}
      - \frac{(n-2r)n^2\,\W_1\W_2}{2(n-r)^3d^3} 
      + \frac{\W_3}{3d^3} \\
      &{\quad} - \frac{r(rn-n+r)\,R_2}{2(n-r)^2nd^2}
       - \frac{r^2n\,R_2\W_1}{2(n-r)^3d^3} + O(r^{-1}n^{-1}d^{-3/5}).
 \end{aligned}
 \end{equation}
Summing (\ref{LgammaW}) over the $\binom{n-1}{r-1}=d/\lambda$ sets $W$ that include $j$, 
 for each $j$, we verify that
 \begin{equation}\label{eq:fred}
       \norm{\varPhi(\bar\gammavec)-\deltavec}_\infty = O(r^{-1}n^{-1}d^{2/5}).
 \end{equation}
 These calculations rely heavily on the identities given in Section~\ref{appendix:j}.

Define $C':= 2^{10} C$, where $C$ is the constant from Lemma~\ref{lem:invjacobian},
and let
 \[ U(C') = \bigl\{ \xvec \st \norm{\xvec-\bar\gammavec}_\infty 
 \le \dfrac{C'}{d} \norm{\varPhi(\bar\gammavec)-\deltavec}_\infty\bigr\}.\]
 Define the function $\nu:\Reals^n\to\Reals^n$ by
 \[\nu(\xvec)=\frac{1}{r}\log\frac{\lambda}{1-\lambda}(1,\ldots,1)\trans +\xvec.\]
 Let $\Psi$ be the function defined in (\ref{eq:Psi-def}). Then for any $\xvec\in \Reals^n$ we have
 $\Psi(\nu(\xvec))=\varPhi(\xvec)-\deltavec$. In particular this implies that 
$J_\varPhi^{-1}(\xvec)=J_{\Psi}^{-1}(\nu(\xvec))$ where $J_\varPhi(\xvec)$ and $J_{\Psi}(\nu(\xvec))$ denote the Jacobians of $\varPhi(\xvec)$ and $\Psi(\nu(\xvec))$ respectively. 
 
We wish to apply Lemma~\ref{lem:invjacobian}.  
Then
\[ \delta_1:= r \max_{j,k\in[n]}|\nu(\bar\gamma)_j - \nu(\bar\gamma)_k|=r \max_{j,k\in[n]}|\bar\gamma_j - \bar\gamma_k|=o(1).\]
Next, using (\ref{eq:fred}),
we have that 
\[ \delta_2:= \frac{r\, C'}{d} \, \norm{\varPhi(\bar\gammavec)-\deltavec}_\infty=o(1).\]
Finally, since $\lambda(\nu(\zerovec))=\lambda\le 1/2$ and 
$\max_{j\in [n]} |\bar\gamma_j|=o(1/r)$,  Lemma~\ref{lem:lambdaW-ratios-different-beta} implies that
\[ e^{\delta_2}\, \lambda(\nu(\bar\gamma))=(1+o(1))\lambda\le 7/8. \]
Hence Lemma~\ref{lem:invjacobian} implies that for every $\xvec\in U(C')$, we have
\[
\norm{J_\varPhi^{-1}(\xvec)}_\infty=\norm{J_{\Psi}^{-1}(\nu(\xvec))}_\infty \le 
  \frac{2^8 C \, e^{o(1)}}{(1+o(1))\, d} < \frac{C'}{d}.
\]
Therefore, by Lemma~\ref{l:Kowa} 
 there exists $\xvec\in U(C')$ such that $\varPhi(\xvec)=\deltavec$. 
 Setting $\gammavec^\ast = \xvec$ proves the lemma, since 
 $\| \xvec - \bar\gammavec\|_\infty = O(r^{-1} n^{-1} d^{-3/5})$ and
 the last term of $\bar\gamma_j$ is
 $O(r^{-1}n^{-1}d^{-3/5})$.
 \end{proof}

Now we can calculate the values of the quantities that appear in
Theorem~\ref{thm:main}.

\begin{lemma}\label{LL}
   Under assumptions~\eqref{mainineq} and~\eqref{nearreg}, we have
   in the first quadrant that
  \begin{align*}
       \prod_{W\in\rsets} &\lambda_W^{\lambda_W}(1-\lambda_W)^{1-\lambda_W}\\
  &{\kern-1em}= \( \lambda^\lambda (1-\lambda)^{1-\lambda}\)^{\binom nr}
     \exp\biggl( \frac{(n-1)\,R_2}{2(1-\lambda)(n-r)d} 
       - \frac{(1-2\lambda)\,R_3}{6(1-\lambda)^2 d^2}
          + \frac{R_4}{12d^3} + O(\deltamax\, d^{-3/5}) \biggr).
 \end{align*}
 \end{lemma}
\begin{proof}
Define $z_W$ by $\lambda_W=\lambda(1+z_W)$ and
\begin{align}
   \eta(z) &= \log \frac{(\lambda(1+z))^{\lambda(1+z)}
                                  (1-\lambda(1+z))^{1-\lambda(1+z)}}
                               {\lambda^\lambda(1-\lambda)^{1-\lambda}}
                   - \lambda z\log\frac {\lambda}{1-\lambda} \nonumber \\
 &= \log\biggl( (1+z)^{\lambda(1+z)}\, \Bigl(1 - \frac{\lambda z}{1-\lambda}\Bigr)^{1-\lambda(1+z)}
     \,\biggr)\nonumber \\
           &= \sum_{j=2}^\infty\,
             \Bigl( \Bigl( \frac{\lambda}{1-\lambda}\Bigr)^{\! j-1}
                + (-1)^j \Bigr) \frac{\lambda}{(j-1)j}\,z^j.\label{eq:eta-coeffs}
\end{align}
Recall that $\sum_{W\in\rsets} z_W = 0$,  therefore,
\begin{equation}\label{LL1}
  \prod_{W\in\rsets} \lambda_W^{\lambda_W}(1-\lambda_W)^{1-\lambda_W}
  = \( \lambda^\lambda (1-\lambda)^{1-\lambda}\)^{\binom nr}
     \exp\biggl(\, \sum_{W\in\rsets} \eta(z_W)\biggr).
\end{equation}
Lemma~\ref{betavalue} implies that $\gamma^\ast_W=\bar{\gamma}_W+O(n^{-1}d^{-3/5})$.
Recalling (\ref{Ltaylor}), this implies that
$L(\gamma_W^\ast)=L(\bar{\gamma}_W)+O(n^{-1}d^{-3/5})$, as $\gamma_W^\ast=o(1)$.
Using~\eqref{LgammaW}, we have
\begin{align}
    z_W &= \frac {(1-\lambda) (e^{\gamma_W^\ast}-1)}{1+\lambda(e^{\gamma_W^\ast}-1)} 
        = (1-\lambda)L_W(\gamma_W^\ast) \notag \\
    &= \frac{(n-1)\,\W_1}{(n-r)d} +\frac{ n(n-2\lambda n-2)\,\W_1^2}{2(1-\lambda)(n-r)^2d^2} 
     +\frac{n^3\,\W_1^3}{6(n-r)^3d^3}
     - \frac{ (n-2\lambda n - 2r)n \,\W_2}{2(1-\lambda)(n-r)^2d^2} \notag\\
    &{\qquad} 
    - \frac{\W_1\W_2}{2d^3} + \frac{\W_3}{3d^3} - \frac{ r^2\,R_2}{2(n-r)^2d^2} + O(n^{-1}d^{-3/5}). 
    \label{zW}
\end{align}
The coefficients of the Taylor expansion of $\eta(z)$ are uniformly $O(\lambda)$,
as shown in (\ref{eq:eta-coeffs}).
Also note that $z_W=O(\deltamax r d^{-1})=O(d^{-1/5})$.  This gives
\begin{align*}
  \eta(z_W) &= \frac{\lambda(n-1)^2\,\W_1^2}{2(1-\lambda)(n-r)^2d^2}
       + \frac{\lambda (n-2\lambda n - 3) n^2\,\W_1^3}{3(1-\lambda)^2(n-r)^3d^3}
       + \frac{\lambda n^4\,\W_1^4}{8(n-r)^4d^4}
       + \frac{\lambda\,\W_2^2}{8d^4}
       + \frac{\lambda\,\W_1\W_3}{3d^4}  \\
   &{\qquad} - \frac{\lambda (n-2\lambda n-2r)n^2\,\W_1\W_2}{2(1-\lambda)^2(n-r)^3d^3}
       - \frac{\lambda\, \W_1^2\W_2}{2d^4}
       - \frac{\lambda r^2n\,R_2\,\W_1}{2(n-r)^3d^3} + 
       O(\lambda r\deltamax n^{-1}d^{-8/5}).
\end{align*}
Using the identities in Section~\ref{appendix:all}, we can sum over all $W\in\rsets$:
\begin{equation}\label{LL2}
    \sum_{W\in\rsets} \eta(z_W) =
       \frac{(n-1)\,R_2}{2(1-\lambda)(n-r)d} 
       - \frac{(1-2\lambda)\,R_3}{6(1-\lambda)^2 d^2}
          + \frac{R_4}{12d^3} + O(\deltamax d^{-3/5}).
\end{equation}
The lemma now follows from~\eqref{LL1} and~\eqref{LL2}.
\end{proof}

Let $A_0$ be the matrix $A$ in the case that $\dvec = (d,d,\ldots, d)$.
That is,
\[  A_0 = \frac{(1-\lambda)(n-r) d}{2(n-1)}\,I + \frac{(1-\lambda)(r-1)d}{2(n-1)}\,J.\]
Then
    \begin{align}
       A_0^{-1} &= \frac{2(n-1)}{(1-\lambda)(n-r)d}\,I - \frac{2(r-1)}{(1-\lambda)r(n-r)d}\,J, \notag\\
       \abs{A_0} &= \frac{ (1-\lambda)^n r (n-r)^{n-1} d^n}{2^n (n-1)^{n-1}}
         = \frac {r\,Q^n}{2^n (n-r)(n-1)^{n-1}},\label{A0det}
   \end{align}
where the determinant follows from (\ref{eq:aIbJ}).   

\begin{lemma}\label{detvalue}
   Under assumptions~\eqref{mainineq} and~\eqref{nearreg}, we have
   in the first quadrant that
  \[
      \abs{A} = \abs{A_0} \,\exp\biggl( -\frac{R_2}{2d^2} + O(\deltamax d^{-3/5}) \biggr).
 \]
\end{lemma}
\begin{proof}
Define the matrix $E$ by $A=A_0+E$.  Then
   \begin{align*}
          A &= A_0 (I-D)^{-1} (I + M), \quad\text{where} \\
        D &:= \diag\biggl( \frac{(1-2\lambda)\delta_1}{(1-\lambda)d},\ldots,
                               \frac{(1-2\lambda)\delta_n}{(1-\lambda)d}\biggr) 
                               \quad\text{and}  \\
        M &:= -D + (I-D) A_0^{-1} E.
  \end{align*}
  For $W\in\rsets$ we have $\lambda_W=\lambda(1+z_W)$, where
  $z_W$ is given by~\eqref{zW}.  This gives
  \[
       \dfrac12\lambda_W(1-\lambda_W) = \dfrac12\lambda(1-\lambda)
       + \frac{\lambda(1-2\lambda)\,\W_1}{2d} + \frac{\lambda\,\W_1^2}{4d^2}
       - \frac{\lambda\,\W_2}{4d^2} + O(\lambda\deltamax n^{-1}d^{-3/5}).
  \]
  Summing over $W\ni j$ and $W\ni j,k$, using Sections~\ref{appendix:j} and~\ref{appendix:jk},
we have $E=(e_{jk})$, where  
  \[
      e_{jk} = \begin{cases}
             \,\displaystyle \dfrac12 (1-2\lambda)\delta_j 
                      + O(\deltamax n^{-1}d^{2/5}), & \text{~if $j=k$;} \\[1ex]
             \,\displaystyle \frac{(1-2\lambda)(r-1)(\delta_j+\delta_k)}{2n}
             + \frac{(r-1)\delta_j\delta_k}{2nd} + O(\deltamax r n^{-2}d^{2/5}), & \text{~if $j\ne k$.}
       \end{cases}
 \]
 This implies that $A_0^{-1}E = (e'_{jk})$, where   
 \[
     e'_{jk} = \begin{cases}
             \,\displaystyle \frac{(1-2\lambda)\delta_j}{(1-\lambda)d}
                      + O(\deltamax n^{-1}d^{-3/5}), & \text{~if $j=k$;} \\[2ex]
             \,\displaystyle \frac{(1-2\lambda)(r-1)\delta_j}{(1-\lambda)nd}
             + \frac{(r-1)\delta_j\delta_k}{nd^2} + O(\deltamax r n^{-2} d^{-3/5}), & \text{~if $j\ne k$.}
       \end{cases}
 \]
Finally, we have $M=(m_{jk})$, where   
 \[
     m_{jk} = \begin{cases}
             \,\displaystyle -\frac{\delta_j^2}{d^2}
                      + O(\deltamax n^{-1}d^{-3/5}), & \text{~if $j=k$;} \\[2ex]
             \,\displaystyle \frac{(1-2\lambda)(r-1)\delta_j}{(1-\lambda)nd}
             - \frac{(r-1)\delta_j^2}{nd^2} + \frac{(r-1)\delta_j\delta_k}{nd^2}
             + O(\deltamax r n^{-2} d^{-3/5}), & \text{~if $j\ne k$.}
       \end{cases}
 \]
To complete the proof, note that 
\[
     \abs{ (I-D)^{-1} } = \prod_{j=1}^n \,
       \biggl( 1 - \frac{(1-2\lambda)\,\delta_j}{(1-\lambda)d} \biggr)^{\! -1}
     = \exp\biggl( \frac{R_2}{2d^2} + O(\deltamax d^{-3/5}) \biggr)
\]
and, since $\norm{M}_2\le \sqrt{\norm{M}_1 \norm{M}_\infty}=o(1)$,
\begin{align*}
   \abs{ I+M }&=\prod_{j=1}^n\,(1+\mu_j)=\exp\biggl(\sum_{j=1}^n (\mu_j +O(|\mu_j|^2))\biggr) 
    = \exp\( \tr M + O(\norm{M}_F^2) \)\\
&= \exp\biggl( -\frac{R_2}{d^2} + O(\deltamax d^{-3/5}) \biggr),	
\end{align*}
where $\mu_1,\ldots,\mu_n$ are the eigenvalues of $M$ and  $\norm{M}_F$ is the Frobenius norm. 
The penultimate equality follows by~\cite[equation (3.71)]{Zhan2002}, which states that
$\sum_{j=1}^n |\mu_j|^2 \le \norm{M}_F^2$.
\end{proof}

\begin{corollary}\label{cor:nearreg}
   Under assumptions~\eqref{mainineq} and~\eqref{nearreg}, we have
   in the first quadrant that
\begin{align*}
   \Hrd &=\frac{r}{2^n\, \pi^{n/2} \, \abs{A_0}^{1/2}} 
\( \lambda^\lambda (1-\lambda)^{1-\lambda} )^{-\binom nr} \\
 &{\qquad}\times \exp\biggl(-\frac{(n-1)\,R_2}{2(1-\lambda)(n-r)d}+\frac{R_2}{4d^2}
  + \frac{(1-2\lambda)\,R_3}{6(1-\lambda)^2d^2}
  - \frac{R_4}{12d^3} + O(\bar\eps)\biggr),
\end{align*}
where $\bar\eps=\eps+ \deltamax d^{-3/5}$
and $\abs{A_0}$ is given by~\eqref{A0det}.
\end{corollary}
\begin{proof}
 This follows by substituting Lemmas~\ref{LL} and~\ref{detvalue}
 into Theorem~\ref{thm:main}.
\end{proof}

Finally, Theorem~\ref{thm:nearreg} removes the assumption of being in the first quadrant.

\begin{proof}[Proof of Theorem~\ref{thm:nearreg}]
Since the formula is invariant under the symmetries
and matches Corollary~\ref{cor:nearreg} within the error term in the first quadrant,
it is true in all quadrants.
To see this, observe that under either of our two symmetries, $R_3$ becomes $-R_3$ and
$(1-2\lambda)(n-2r)$ becomes $-(1-2\lambda)(n-2r)$.
\end{proof}

\section{Degrees of random uniform hypergraphs}\label{s:degree-models}

We now show how to apply the results of Section~\ref{s:nearreg} to analyse the
degree sequence of a random uniform hypergraph with a given number of edges.
 Define $B(K,x)=\binom{K}{\lambda K+x}$ where $K$, $\lambda K + x$ are integers.
The following lemma is a consequence of Stirling's expansion 
for the gamma function. 

\begin{lemma}\label{binstirling}
 Let $K,x,\lambda$ be functions of $n$ such that, as $n\to\infty$,
 $\lambda\in(0,1)$, $\lambda(1-\lambda)K\to\infty$ and $x=o\(\lambda(1-\lambda)K\)$.
 Then
 \begin{align*}
    B(K,x) &=
        \frac{\lambda^{-\lambda K-x-1/2} \,(1-\lambda)^{-(1-\lambda)K+x-1/2}}
               {\sqrt{2\pi K}} \\
    &{\quad}\times\exp\biggl( -\frac{x^2}{2\lambda(1-\lambda)K}
           - \frac{(1-2\lambda)x}{2\lambda(1-\lambda)K}
           - \frac{1-\lambda+\lambda^2}{12\lambda(1-\lambda)K}
           + \frac{(1-2\lambda)x^3}{6\lambda^2(1-\lambda)^2K^2} \\
    &{~\qquad\qquad} + \frac{(1-2\lambda+2\lambda^2)x^2}{4\lambda^2(1-\lambda)^2K^2}
           + \frac{(1-2\lambda)x}{12\lambda^2(1-\lambda)^2K^2}
           - \frac{(1-3\lambda+3\lambda^2)x^4}{12\lambda^3(1-\lambda)^3K^3} \\
       &{~\qquad\qquad}    + O\Bigl( \frac{\abs{x}^3+1}{\lambda^3(1-\lambda)^3K^3}
            + \frac{\abs{x}^5}{\lambda^4(1-\lambda)^4K^4} \Bigr)
            \biggr).
 \end{align*}
 \end{lemma}

\begin{proof}
This follows from Stirling's expansion for the factorial, which we use in the form
\[ N! = \sqrt{2\pi}\, N^{N+1/2} e^{-N}\
   \exp\biggl( \frac{1}{12N} +  O(N^{-3})\biggr).
\]
From this we obtain
\begin{align*}
  B(K,x) &= \frac{K^{K+1/2}} 
                    {\sqrt{2\pi}\,  (\lambda K+x)^{\lambda K+x+1/2} ((1-\lambda)K-x)^{(1-\lambda)K-x+1/2}} \\
             &{\quad}\times\exp\biggl( \frac{1}{12K} - \frac{1}{12(\lambda K+x)}
                         - \frac{1}{12((1-\lambda)K-x)} 
                         + O\biggl(\frac {1}{\lambda^3(1-\lambda)^3K^3}\biggl) \biggl)
\end{align*}
Now write
\[
      (\lambda K+x)^{\lambda K+x+1/2} = (\lambda K)^{\lambda K+x+1/2}
     \exp\biggl( \(K+x+\dfrac12\) \log\biggl(1 + \frac{x}{\lambda K}\biggl)\biggl)
\]
and similarly for $((1-\lambda)K-x)^{(1-\lambda)K-x+1/2}$.
Expanding the logarithms gives the desired result.
\end{proof}

\bigskip

\begin{proof}[Proof of Theorem~\ref{thm:BvsD}]
 For some $p\in (0,1)$, let $X_1,\ldots, X_n$ be iid random variables with the binomial
    distribution $\Bin\(\binom{n-1}{r-1},p\)$.
    Then $\calB_r(n,m)$ is the distribution of $(X_1,\ldots,X_n)$ conditioned on the
    sum being $nd$.
    Since the sum has distribution $\Bin\(n\binom{n-1}{r-1},p\)$, we find that the
    conditional probability is independent of~$p$\,:
    \[
         \Prob_{\calB_r(n,m)}(\dvec) = \binom{n\binom{n-1}{r-1}}{nd}^{\!-1}
            \, \prod_{j=1}^n\, \binom{ \binom{n-1}{r-1}}{d_j}.
   \]
  Consequently,
   \[
      \frac{\Prob_{\calD_r(n,m)}(\dvec)}{\Prob_{\calB_r(n,m)}(\dvec)} =
      \frac{B\(n\binom{n-1}{r-1},0\)  H_r(\dvec)}
             {B\(\binom nr,0\)
               \prod\nolimits_{j=1}^{n} B\(\binom{n-1}{r-1},\delta_j\)}.
  \]
  Now use Theorem~\ref{thm:nearreg} for $H_r(\dvec)$ and
  Lemma~\ref{binstirling} for the other factors. 
\end{proof}

\smallskip
Let $Z_1,\ldots, Z_n$ be iid random variables having the hypergeometric
   distribution with parameters $\binom nr,m,\binom{n-1}{r-1}$,
   where $m=e(\dvec)$.  That is,
   \begin{equation}\label{hgdef}
      \Prob(Z_j=k) = \binom{\binom nr}{m}^{\!-1}
      \binom{\binom{n-1}{r-1}}{k} \binom{\binom nr-\binom{n-1}{r-1}}{m-k} .
  \end{equation}
Note that $Z_1$ has precisely the distribution of the degree of one vertex
in a uniformly random $r$-uniform hypergraph with $n$ vertices and $m$ edges.
Now let $\calT_r(n,m)$ be the distribution of $Z_1,\ldots,Z_n$ when
conditioned on having sum~$nd$.
If $P:=\Prob(Z_1+\cdots+ Z_n=nd)$, for which there seems to be no closed formula,
we have
\begin{equation}\label{Trdef}
      \Prob_{\calT_r(n,m)}(\dvec) = P^{-1} \binom{\binom nr}{m}^{\!-n}
      \prod_{j=1}^n \;\Biggl( 
         \binom{\binom{n-1}{r-1}}{d_j} \binom{\binom nr-\binom{n-1}{r-1}}{m-d_j}
      \Biggr).
\end{equation}

 \smallskip

 \begin{lemma}\label{hyperbounds}
 Let $Z_1,\ldots,Z_n$ be independent hypergeometric variables with
 distribution given by~\eqref{hgdef} and let
 $X_1,\ldots,X_n$ be the same conditioned on $\sum_{j=1}^n Z_j=nd$. Then
 \begin{itemize}\itemsep=0pt
   \item[\emph{(a)}] Each $Z_j$ and $X_j$ has mean $d$. Also, $Z_j$ has variance
         \begin{equation}\label{variance}
        \sigma^2 = \frac{(1-\lambda)(n-r)d^2}{nd-\lambda r}
                       = \frac{Q}{n}\biggl(1 - \binom nr^{\!\!-1}\,\biggr)^{\!-1}.
          \end{equation}
  \item[\emph{(b)}] For $t\ge 0$, we have for any $j$ that
    \[
        \Prob(\abs{Z_j-d} \ge t) \le 2\exp\biggl( -\frac{t^2}{2(d+t/3)}\biggr)
        \le \begin{cases} 2\exp\Bigl(-\dfrac {t^2}{4d}\Bigr), & 0\le t\le 3d;\\
                                 2e^{-3t/4}, & t\ge 3d.
             \end{cases}
    \]
   \item[\emph{(c)}] If $nd+y$ is an integer in $[0,mn]$, then
    \[
        \Prob\Bigl(\, \sum\nolimits_{j=1}^n Z_j=nd+y\Bigr) = 
           \frac{1}{\sigma \sqrt{2\pi n}} \exp\biggl(-\frac{y^2}{2n\sigma^2}\biggr)
             + O(n^{-1}\sigma^{-2}),
    \]
    where the implicit constant in the error term is bounded absolutely.
   \item[\emph{(d)}] For every nonnegative integer $y$, $\Prob(X_1=y)=C(y)\Prob(Z_1=y)$, where uniformly
    \[
     C(y) =  \frac{\Prob\(\sum_{j=2}^n Z_j = nd-y\)}{\Prob\(\sum_{j=1}^n Z_j=nd\)}
     = (1+O(n^{-1})) \exp\biggl( -\frac{(y-d)^2}{2(n-1)\sigma^2}\biggr) + O(n^{-1/2}\sigma^{-1}).
  \] 
    \item[\emph{(e)}] If $\sigma^2\geq 1$ then for $t>0$,
    \begin{align*}
        \E \min\{ (Z_1-d)^2,t^2\} &= \sigma^2  + O\(e^{-t^2/(4d)}d + e^{- 9d/4}d\), \\
        \E \min\{ (X_1-d)^2,t^2\} &= (1+O(n^{-1}))\, \sigma^2  + O\(e^{-t^2/(4d)}d + e^{- 9d/4}d\).
    \end{align*}
  \end{itemize}
 \end{lemma}
 \begin{proof}
    Part (a) is standard theory of the hypergeometric distribution.
    For parts (b) and (c), we note that Vatutin and Michailov~\cite{VM}
    proved that $Z_j$ can be expressed as the sum of $m$ independent
    Bernoulli random variables (generally with different means).
    Inequality (b) is now standard (see~\cite[Theorem~2.1]{JansonRG}),
    while (c) was proved by Fountoulakis, Kang and Makai~\cite[Theorem~6.3]{FKM}.
    
For part (d), the standard formula for conditional probability implies that
the expression for $\Prob(X_1=y)$ holds with
     $C(y) =  \dfrac{\Prob\(\sum_{j=2}^n Z_j = nd-y\)}{\Prob\(\sum_{j=1}^n Z_j=nd\)}$.
Then by part (c) we have
\begin{align*}
\Pr\Bigl(\sum_{j=2}^n Z_j = nd-y\Bigr) &= \frac{1}{\sigma\sqrt{2\pi(n-1)}}\, \exp\left(\frac{-(y-d)^2}{2(n-1)\sigma^2}\right) + O(n^{-1}\sigma^{-2}),\\
\Pr\Bigl(\sum_{j=1}^n Z_j = nd\Bigr) &= \frac{1}{\sigma\sqrt{2\pi n}}\(1 + O(n^{-1/2}\sigma^{-1})\),
\end{align*}
and dividing the first expression
by the second gives the stated approximation for $C(y)$.

    For (e), we have
    \[
        \E \min\{ (Z_1-d)^2,t^2\} = \sigma^2 -
          \sum_{\abs\ell>t}\, (\ell^2 - t^2)\, \Prob(Z_1=d+\ell),
    \]
    where the sum is restricted to integer $d+\ell$.  We will consider the
    upper tail, noting that the lower tail is much the same:
    \begin{align*}
       \sum_{\ell>t}\, (\ell^2 - t^2)\, \Prob(Z_1=d+\ell)
         &= \sum_{\ell>t}\, (\ell^2-t^2)\( \Prob(Z_1\ge d+\ell)-\Prob(Z_1\ge d+\ell+1)\) \\
         &\le (2t+1)\Prob(Z_1\ge d+t) + \sum_{\ell>t}\, (2\ell+1)\, \Prob(Z_1\ge d+\ell+1).
    \end{align*}
    Now we can use the first case of part (b) to obtain the bound
    $O(e^{-t^2/(4d)}d)$ and the second case to obtain the bound $O(e^{-9d/4}d)$.

    For the second part of (e),
we have
\begin{align*}
\E\((X_1-d)^2\) &= \sigma^2 + \sum_j \( C(j)-1 \)\, \Pr(Z_1=j)\, (j-d)^2\\
                &= \sigma^2 + \sum_j \left(\exp\left(-\frac{(j-d)^2}{2(n-1)\sigma^2}\right) - 1 + O(1/n)\right)\, \Pr(Z_1=j)\, (j-d)^2\\
    &= \sigma^2\(1 + O(n^{-1})\) + O\biggl(\frac{\E\( (Z_1 - d)^4\)}{n\sigma^2}\biggr).
\end{align*}
Since $\sigma^2\geq 1$, the fourth central moment of $Z_1$ satisfies $\E\((Z_1-d)^4\)
=O(\sigma^4)$,
as follows from the exact expression given in~\cite[equation (5.55)]{KS}.
Therefore
\[ \E\((X_1-d)^2\) = \sigma^2\(1 + O(n^{-1})\).\]
Then the effect of truncation at $t$ can be 
    bounded as before, using the fact that $C(\ell)=\nobreak O(1)$.
 \end{proof}
 
\begin{proof}[Proof of Theorem~\ref{thm:DvsT}]
 From the definitions of $\calD_r(n,m)$ and~\eqref{Trdef}, we have
 \[
      \frac{\Prob_{\calD_r(n,m)}(\dvec)}{\Prob_{\calT_r(n,m)}(\dvec)}
      = 
      \frac{B\(\binom nr,0\)^{n-1}\,P\,H_r(\dvec)}
      {\prod_{j=1}^n \Bigl( B\(\binom{n-1}{r-1},\delta_j\) 
            B\(\binom nr-\binom{n-1}{r-1},-\delta_j\)\Bigr)}.
 \]
 Now use Theorem~\ref{thm:nearreg} for $H_r(\dvec)$,
 Lemma~\ref{hyperbounds}(c) for $P$, and Lemma~\ref{binstirling}
 for the other factors. 
\end{proof}

For the proof of Theorem~\ref{thm:Tconj} we need a concentration lemma.

\begin{lemma}\label{concen}
  Let $f(x_1,\ldots,x_K):\{0,1\}^K\to\Reals$ be a function such that
  $\abs{f(\xvec)-f(\xvec')}\le a$ whenever $\xvec,\xvec'$ differ in only
  one coordinate.
  Let $\Zvec=(Z_1,\ldots,Z_K)$ be independent Bernoulli variables
  (not necessarily identical),
  conditioned on having constant sum~$S$.
  Then, for any $t\ge 0$,
  \[
      \Prob\( \abs{f(\Zvec)-\E f(\Zvec)} > t\) \le 2\exp\biggl( -\frac{t^2}{8a^2 S}\biggr).
  \]
\end{lemma}
\begin{proof}
 According to Pemantle and Peres~\cite[Example~5.4]{PP}, the measure defined
 by independent
 Bernoulli variables conditioned on a fixed sum has the ``strong Rayleigh'' property.
The proof is completed by applying~\cite[Theorem~3.1]{PP}.
\end{proof}

\begin{proof}[Proof of Theorem~\ref{thm:Tconj}] 
Probabilities in the hypergeometric distribution are symmetric under the
two operations (that is, replacing $r$ by $n-r$, or replacing $m$ by $\binom{n}{r}-m$).
Since the error term given in the theorem is also symmetric under
these operations, it suffices to assume that $(r,\dvec)$ belongs to the first quadrant.

Define
\[
   R_2(\dvec) := \sum_{j=1}^n\,(d_j-d)^2 \quad\text{and}\quad
   R'_2(\dvec) := \sum_{j=1}^n\,\min\{ (d_j-d)^2, d\log^2 n\},
\]
and
\[
   \frakW := \bigl\{ \dvec \st \deltamax\le d^{1/2}\log n\text{~and~}
                  \abs{R_2(\dvec)-n\sigma^2}\le n^{1/2}\sigma^2\log^2 n\bigr\}.
\]
Let $Z_1,\ldots,Z_n$ be iid random variables with distribution~\eqref{hgdef}. 
The distribution $\calT_r(n,m)$ is that of $(Z_1,\ldots,Z_n)$ conditioned on $\sum_{j=1}^n Z_j=nd$.
    
By the union bound,  we have
  \begin{align*}
     \Prob_{\calT_r(n,m)}\(\abs{R_2(\dvec)&-n\sigma^2} > n^{1/2}\sigma^2\log^{2}n\) 
      \le \Prob_{\calT_r(n,m)}\(R_2(\dvec)\ne R'_2(\dvec)\) \\ 
         &{\quad}+ \Prob_{\calT_r(n,m)}\(\abs{R'_2(\dvec)-\E R'_2(\dvec)} 
             >n^{1/2}\sigma^2\log^2 n - \abs{n\sigma^2-\E R'_2(\dvec)}\).
  \end{align*}
      Since always $C(i)=O(1)$, Lemma~\ref{hyperbounds}(b,d) and the union bound give
  \[
       \Prob_{\calT_r(n,m)}\(R_2(\dvec)\ne R'_2(\dvec)\) \le 
       n \sum_{i\mathrel{:}\abs{i-d}>d^{1/2}\log n} \Prob_{\calT_r(n,m)}(Z_1=i)\, C(i)
       = n^{-\Omega(\log n)}.
  \]
Next, note that in $\calT_r(n,m)$ we have
  $\abs{n\sigma^2-\E R'_2(\dvec)} =O(\sigma^2)=O(d)$ by Lemma~\ref{hyperbounds}(e);
  for later use note that this only relies on the condition $\deltamax\le d^{1/2}\log n$.
  Recall that each $Z_j$ is the sum of $m$ independent Bernoulli variables,
  so $R'_2(\dvec)$ is a function of $mn$ independent Bernoulli variables
  conditioned on fixed sum~$nd$.
Changing one of the Bernoulli variables changes the corresponding $d_j$ by one
and changes $d$ by $1/n$.
Overall, this changes the value of 
$R'_2(\dvec)$ by at most $2+4d^{1/2}\log n$. 
Applying Lemma~\ref{concen}, we have
  \begin{equation}\label{R2tail}
        \Prob_{\calT_r(n,m)}\(\abs{R'_2(\dvec)-\E R'_2(\dvec)} >n^{1/2}\sigma^2\log^2 n
           - \abs{n\sigma^2-\E R'_2(\dvec)}\) = n^{-\Omega(\log n)}.
  \end{equation}
 Therefore, $\Prob_{\calT_r(n,m)}(\frakW) = 1 - n^{-\Omega(\log n)}$.
 Now we can apply Theorem~\ref{thm:DvsT} to obtain
 \[
     \Prob_{\calD_r(n,m)}(\dvec)=\(1 + O(\eps+n^{1/10}Q^{-1/10}\log n + n^{-1/2}\log^2 n)\)
        \Prob_{\calT(n,m)}(\dvec)
 \]
 for $\dvec\in\frakW$.  Here, $\eps$ and $n^{1/10}Q^{-1/10}\log n$ come
 from the error terms in Theorem~\ref{thm:DvsT}, while $n^{-1/2}\log^2 n$ comes
 from the term $R_2/Q$ in Theorem~\ref{thm:DvsT} since $n\sigma^2=Q(1+O(n^{-1/2}\log^2 n))$
 in~$\frakW$, by the definition of $\frakW$ and~\eqref{variance}.

Now consider the probability space $\calD_r(n,m)$.
Since the distribution of each individual degree is the same as the distribution of~$Z_1$,
using a union bound and applying Lemma~\ref{hyperbounds}(b) gives
$\Prob_{\calD_r(n,m)} (\deltamax > d^{1/2}\log n) = n^{-\Omega(\log n)}$ and
hence
\[
     \Prob_{\calD_r(n,m)} (R_2(\dvec)\ne R'_2(\dvec)) = n^{-\Omega(\log n)}.
\]
In~\cite{KLW}, concentration of $R_2(\dvec)$ in $\calD_r(n,m)$ is shown using a lemma
on functions of random subsets. However, that approach (at least, using the same
concentration lemma) apparently only works for $r=o(n/\log n)$, so we will adopt a different
approach.

By the same argument as used to prove~\eqref{R2tail},
\[
     \Prob_{\calT_r(n,m)}\(\abs{R_2(\dvec)-n\sigma^2} >k n^{1/2}d\log^2 n 
        \bigm|  \deltamax\le d^{1/2}\log n \) \le e^{-C k^2\log^2 n}
\]
for any positive integer $k$ and some constant $C>0$ independent of~$k$.
(We have used $\abs{n\sigma^2-\E R'_2(\dvec)} =O(d)$ as before.)

If $R_2(\dvec)\leq (k+1)\, n^{1/2}\sigma^2\log^2 n$ then
$-\dfrac{1}{2} + \dfrac{R_2(\dvec)}{2Q} \leq \dfrac{(k+1)\log^2 n}{2n^{1/2}} + o(1)$
and so applying Theorem~\ref{thm:DvsT} gives 
\begin{align*}
     \Prob_{\calD_r(n,m)}\(k n^{1/2}\sigma^2\log^2 n
         < \abs{R_2(\dvec)&-n\sigma^2} \le (k+1) n^{1/2}\sigma^2\log^2 n 
      \bigm| 
\deltamax\le d^{1/2}\log n \) \\
     & \le \exp\Bigl( -Ck^2\log^2 n + \frac{(k+1)\log^2 n}{2n^{1/2}} + o(1) \Bigr).
\end{align*}
Summing over $k\ge 1$, we have
\[
     \Prob_{\calD_r(n,m)}\(\abs{R_2(\dvec)-n\sigma^2} >n^{1/2}\sigma^2\log^2 n 
   \bigm| \deltamax\le d^{1/2}\log n \) = n^{-\Omega(\log n)},
\]
and therefore $\Prob_{\calD_r(n,m)}(\frakW)=1-n^{-\Omega(\log n)}$,
completing the proof.
\end{proof}

\section{Deferred proofs}\label{s:technical}

\subsection{Proof of Lemma~\ref{lem:symmetries}}\label{s:proof-symmetries}

We begin with the operation of replacing each edge by its complement in~$V$,
which sends $d_j$ to $d'_j=e(\dvec)-d_j$ for each~$j$.
Recall that 
\[\beta'_j= \frac{1}{n-r} \biggl(\, \sum_{\ell\in [n]} \beta^\ast_\ell \biggr)  - \beta^\ast_j\]
and note that for all $j,k\in [n]$,
\[
  |\beta'_j -\beta'_k|= \biggl|\,\frac{1}{n-r} \biggl(\, \sum_{\ell\in [n]} \beta^\ast_\ell \biggr) 
   - \beta^\ast_j - \frac{1}{n-r}\biggl(\, \sum_{\ell\in [n]} \beta^\ast_\ell \biggr)  + \beta^\ast_k \,\biggr|=|\beta^\ast_j-\beta^\ast_k|.
\]
In addition, for any $W\in \rsets$ we have
\[  \sum_{j\in V\setminus W}  \beta'_j  = \frac{n-r}{n-r} \biggl(\, \sum_{\ell\in [n]} \beta^\ast_\ell \biggr)  - \sum_{j \in V\setminus W} \beta^\ast_j
= \sum_{j \in W} \beta^\ast_j.\]
Therefore for any $W\in\rsets$ we have
\begin{equation}\label{eq:lambdacompl}
\lambda_{V\setminus W}(\betavec')=\frac{e^{\sum_{k\in V\setminus W}\beta_k'}}{1+e^{\sum_{k\in V\setminus W}\beta_k'}}=\frac{e^{\sum_{k\in W}\beta^\ast_k}}{1+e^{\sum_{k\in W}\beta^\ast_k}}=\lambda_W(\betavecstar).
\end{equation}
Note that summing~\eqref{exact} over~$j$ each edge is counted $r$ times, so
$\sum_{W\in\rsets} \lambda_W(\betavecstar)=e(\dvec)$. Hence 
\[ 
\sum_{\substack{W\ni j\\ W\in \mathcal{S}_{n-r}(n)}} \!\!\lambda_W(\betavec')\stackrel{\eqref{eq:lambdacompl}}{=}\sum_{\substack{W\not \ni j\\ W\in \mathcal{S}_{r}}}\lambda_W(\betavecstar)=\sum_{W\in\rsets}\lambda_W(\betavecstar)-\sum_{\substack{W \ni j\\ W\in \mathcal{S}_{r}}}\lambda_W(\betavecstar) = e(\dvec)-d_j,
\]
proving that $(\dvec',\betavec')$ satisfies~\eqref{exact}.
It only remains to show that
\begin{equation}
\label{eq:detA-identity}
\frac{|A(\betavec')|}{(n-r)^2}=\frac{|A(\betavecstar)|}{r^2}.
\end{equation}
For $W\subseteq[n]$, define the $n\times n$ matrix $\varXi_W$ by
\[
     (\varXi_W)_{jk} = \begin{cases} \,1,& \text{~if $j,k\in W$;} \\
                                                    \,0,& \text{~otherwise}.
                             \end{cases}
\]
Then,
\[
   A(\betavecstar) = \sum_{W\in\rsets} \lambda_W(\betavecstar)(1-\lambda_W(\betavecstar))\,\varXi_W.
\]
Now note that $(I-\frac1r J)\varXi_W(I-\frac1r J)=\varXi_{V\setminus W}$ for any~$W\in\rsets$.
(The case $W=\{1,\ldots,r\}$ is representative and easy to check.)
Together with~\eqref{eq:lambdacompl}, this proves that
\[ \Bigl(I-\dfrac1r J\Bigr) \, A(\betavecstar) \, \Bigl(I-\dfrac1r J\Bigr) = A(\betavec').\]
Finally, $\abs{I-\frac1r J} = -\dfrac{n-r}{r}$ by (\ref{eq:aIbJ}), which proves (\ref{eq:detA-identity}).

\medskip
Next, consider the operation that complements the edge set, sending
$d_j$ to $\medtilde d_j=\binom{n-1}{r-1}-d_j$ without changing the edge size.
Recall that $\medtilde\beta_j=-\beta_j^\ast$ for each~$j$.  Then
$|\medtilde\beta_j - \medtilde\beta_k| = |\beta_j^\ast - \beta_k^\ast|$ for all $j,k$.
Note that for any $W\in \rsets$ we have
\[
\lambda_W(\medtilde\betavec)=\frac{e^{\sum_{k\in W}\medtilde\beta_k}}{1+e^{\sum_{k\in W}\medtilde\beta_k}}
=\frac{e^{-\sum_{k\in W}\beta^\ast_k}}{1+e^{-\sum_{k\in W}\beta^\ast_k}}=1-\lambda_W(\betavecstar),
\]
which implies that $A(\medtilde\betavec)=A(\betavecstar)$. 
In addition,
\[\sum_{\substack{W\ni j}}\lambda_W(\medtilde\betavec)
=\binom{n-1}{r-1}-\sum_{\substack{W\ni j}}\lambda_W(\betavecstar)=\binom{n-1}{r-1}-d_j
=\medtilde d_j,
\]
proving that $(\medtilde\dvec,\medtilde\betavec)$ satisfies~\eqref{exact}.
 
 \medskip
 
 The third operation, which complements both the edges and the edge set simultaneously,
 is just the composition of the first two in either order. Hence the result for this
operation follows immediately, completing the proof.

\subsection{Proof of Lemma~\ref{lem:diaggeneral-r}}\label{s:proof-diaggeneral-r}

The following lemmas will be useful.

\begin{lemma}[{\cite[(1.13)]{Higham}}]\label{rank1pwr}
	For $p \in\Reals$, define
	\[
	\alpha_p(x) := \frac{(1+x^2)^p-1}{x^2}.
	\]
	Then, for $\xvec\in\Reals^n$,
	\[
	(I + \xvec\xvec\trans)^p = I + \alpha_p(\norm{\xvec}_2)\, \xvec\xvec\trans.
	\]
	Also, for $x\geq 0$,
	$\abs{\alpha_{-1/2}(x)} \leq x^{-2}$
	and $\abs{\alpha_{1/2}(x)} \leq x^{-1}$.
\end{lemma}

For a matrix $X=(x_{jk})$, $\maxnorm{X}:= \max_{j,k} \abs{x_{jk}}$ is a matrix
norm that is not submultiplicative.
The following is a special case of a lemma in~\cite{mother}.

\begin{lemma}[{\cite[Lemma~4.9]{mother}}]\label{diagonal}
	Let $M$ be a real symmetric positive definite $n\times n$
	matrix with 
	\[
	\maxnorm{M-I} \le \frac{\kappa }{n} \ \
	\text{~~and~~} \ \
	\xvec\trans\! M\xvec \ge \gamma\, \xvec \trans\xvec
	\]
	for some $1\geq \gamma >0$, $\kappa>0$ and all $\xvec\in\Reals^n$.
	Then the following are true.
	\begin{itemize}\itemsep=0pt
		\item[\emph{(a)}] 
		\[
		\maxnorm{ M^{-1} - I} \leq \frac{(\kappa+\gamma) \kappa  }{\gamma n}.
		\]
		
		\item[\emph{(b)}] There exists a real matrix $T$ such that $T\trans\! M T=I$ and
		\[ \norm{T}_1, \norm{T}_\infty 
		\le \frac{\kappa+\gamma^{1/2}}{\gamma^{1/2}}, 
		\ \ \ \ 
		\norm{T^{-1}}_1, \norm{T^{-1}}_\infty
		\le \frac{(\kappa+1)(\kappa+\gamma^{1/2})}{\gamma^{1/2}}. \]
	\end{itemize}
	
\end{lemma}

The next result will be used to find a change of basis matrix to
invert $A(\betavec)$.

\begin{lemma}\label{motherL49}
	Let $\bar{A}=D+\svec\svec\trans+X$ be a symmetric positive definite real
	matrix of order~$n$, where $D$ is a positive diagonal matrix and
	$\svec\in\Reals^n$.
	Define these quantities:
	\begin{align*}
		\gamma &:= \text{a value in }(0,1)\text{ such that }
		\xvec\trans\! \bar{A}\xvec\geq\gamma\,\xvec\trans (D+\svec\svec\trans)\xvec
		\text{ for all }\xvec\in\Reals^n, \\
		\Dmin,\Dmax &:= \text{the minimum and maximum
			diagonal entries of }D, \\
		B &:=1+\Dmax \Dmin^{-1} \norm{\svec}_1\norm{\svec}_{\infty}\norm{\svec}_2^{-2},\\
		\kappa &:= B^2\Dmin^{-1}\, n\, \maxnorm{X}.
	\end{align*}
	Then there is a real $n\times n$ matrix $T$ such that
	$T\trans\! \bar A T=I$ and the following are true:
	
	\begin{itemize}\itemsep=0pt
		\item[\emph{(a)}]
		\[
		\maxnorm{\bar{A}^{-1}-(D+\svec\svec\trans)^{-1}}
		\leq \frac{B^2\kappa(\kappa+1)}
		{\Dmin\gamma n},
		\text{ ~where}
		\] 
		\[
		(D+\svec\svec\trans)^{-1} =
		D^{-1} - \frac{D^{-1}\svec\svec\trans D^{-1}}
		{1 + \norm{D^{-1/2}\svec}_2^2};
		\]
		\item[\emph{(b)}]
		\[
		\norm{T}_1, \norm{T}_\infty \leq
		B \Dmin^{-1/2}\gamma^{-1/2}(\kappa+1);
		\]
		\item[\emph{(c)}]  For any $\rho>0$, define 
		\[
		\calQ(\rho):= U_n(\rho) \cap 
		\biggl\{ \xvec\in\Reals^n \st \abs{\svec\trans\xvec}  \leq
		\frac{\Dmax\norm{\svec}_1}
		{ \Dmin^{1/2}\norm{\svec}_2} \rho \biggr\}.
		\]
		Then
		\[
		T\( U_n(\rho_1)\) \subseteq \calQ(\rho) \subseteq T\(U_n(\rho_2)\),
		\]
		where 
		\begin{align*}
			\rho_1:=  \dfrac{1}{B} \, \Dmin^{1/2}\, \gamma^{1/2}\, (\kappa + 1)^{-1}\rho, \qquad
			\rho_2:= B \Dmax^{\, 1/2}\, \gamma^{-1/2}\, (\kappa + 1)^2\rho.
		\end{align*}
	\end{itemize}
\end{lemma}

\begin{proof}
	Define $\svec_1:= D^{-1/2}\svec$, $X_1:=D^{-1/2}XD^{-1/2}$, $T_1:=(I+\svec_1\svec_1\trans)^{-1/2}$
	and $X_2:=T_1\trans\,X_1\, T_1$.
	By Lemma~\ref{rank1pwr}, we have
	\begin{equation}\label{eq:T1expand}
	T_1=
	I+\alpha_{-1/2}(\norm{\svec_1}_2) \svec_1\svec_1\trans,
	\end{equation}
	and note that $T_1$ is symmetric, that is, \ $T_1=T_1\trans$.
	Therefore
	\begin{equation}\label{eq:transform}
	\bar{A}=D+\svec\svec\trans+X=D^{1/2}\left(I+\svec_1\svec_1\trans+X_1\right)D^{1/2}=D^{1/2}T_1^{-1}\left(I+X_2\right)T_1^{-1}D^{1/2}.	
	\end{equation}
    Recall that by Lemma~\ref{rank1pwr} we have $|\alpha_{-1/2}(\norm{\svec_1}_2)|\le \norm{\svec_1}_2^{-2}$, so by \eqref{eq:T1expand},
	\begin{equation}\label{T1norms}
		\norm{T_1}_1,\,\, \norm{T_1}_\infty 
		\leq 1 + \frac{\norm{\svec_1}_1\norm{\svec_1}_\infty}
		{\norm{\svec_1}_2^2} \leq 1+\frac{\Dmax \norm{\svec}_1\norm{\svec}_{\infty}}{\Dmin \norm{\svec}_2^2}=B.
	\end{equation}
	
	Next we apply Lemma~\ref{diagonal} with $M=I+X_2$.
	By \eqref{eq:transform}, 
$\xvec\trans \bar A\xvec \geq 
	\gamma\xvec\trans(D+\svec\svec\trans)\xvec$
	is equivalent to 
	\[
	(T_1^{-1}D^{1/2} \xvec)\trans \, (I+X_2)T_1^{-1}D^{1/2}\xvec
	 \ge \gamma\, (T_1^{-1}D^{1/2}\xvec)\trans\, T_1^{-1}D^{1/2}\xvec
	\]
	 for all $\xvec\in \Reals^n$.
	Also  \[ \maxnorm{X_2}\leq\Dmin^{-1}\norm{T_1}_\infty^2\maxnorm{X}
	\stackrel{(\ref{T1norms})}{\leq} B^2\Dmin^{-1}\maxnorm{X} = \frac{\kappa}{n}. \]
         Therefore, $M,\gamma,\kappa$ satisfy the conditions of Lemma~\ref{diagonal}.
	Consequently, there exists a transformation
$T_2$ such that $T_2\trans(I+X_2)T_2=I$. This, together with \eqref{eq:transform} implies that
	$T=D^{-1/2}T_1T_2$ satisfies $T\trans \bar AT=I$.
	In addition, by Lemma~\ref{diagonal}(b), we have
	\begin{equation}\label{T2norms}
		\norm{T_2}_1,\norm{T_2}_\infty \leq \gamma^{-1/2}(\kappa+1),
		\qquad
		\norm{T_2^{-1}}_1,\norm{T_2^{-1}}_\infty \leq \gamma^{-1/2}(\kappa+1)^2.
	\end{equation}
	Together with \eqref{T1norms} and $\norm{D^{-1/2}}_1,\norm{D^{-1/2}}_\infty\le \Dmin^{-1/2}$, this proves part (b).
	
	Next we prove the first inclusion of part~(c).  Let $\xvec\in U_n(\rho_1)$,
	that is, $\norm{\xvec}_\infty\leq \rho_1$.
	Then $\norm{T\xvec}_\infty\leq\norm{T}_\infty \, \rho_1 \leq \rho$ by part~(b),
	so $T \xvec\in U_n(\rho)$.
	Next
	\[
	\abs{\svec\trans T\xvec} = \abs{\svec_1\trans T_1 T_2\xvec}
	\leq \norm{T_1\svec_1}_1 \norm{T_2\xvec}_\infty.
	\]
	From~\eqref{T2norms}, $\norm{T_2\xvec}_\infty\leq 
	\gamma^{-1/2}( \kappa+1)\rho_1$.
	Also \eqref{eq:T1expand} gives
	$T_1\svec_1=(1+\norm{\svec_1}_2^2)^{-1/2}\svec_1$,
	so
	\[ \norm{T_1\svec_1}_1\leq \norm{\svec_1}_1 \norm{\svec_1}_2^{-1}\le \Dmax^{1/2} \norm{\svec}_1 \Dmin^{-1/2} \norm{\svec}_2^{-1}.\]
	Combining these bounds proves the inclusion, as $B\ge 1$.
	
	For the second inclusion of part~(c), consider $\xvec\in \calQ(\rho)$.
	Lemma~\ref{rank1pwr} implies that $T_1^{-1}=I+\alpha_{1/2}(\norm{\svec_1}_2) \svec_1\svec_1\trans$, and hence
	\[
	\norm{T^{-1}\xvec}_\infty = \norm{T_2^{-1}T_1^{-1}D^{1/2}\xvec}_\infty
	\leq \norm{T_2^{-1}}_\infty
	\Norm{D^{1/2}\xvec
		+ \alpha_{1/2}(\norm{\svec_1}_2)\svec_1 \svec\trans\xvec}_\infty.
	\]
	Now apply~\eqref{T2norms} to $\norm{T_2^{-1}}_\infty$,
	the first part of the definition of $\calQ(\rho)$ to $\norm{D^{1/2}\xvec}_\infty$,
	the second part of the definition of $\calQ(\rho)$ to $\abs{\svec\trans\xvec}$,
	and recall from Lemma~\ref{rank1pwr} that
	$\abs{\alpha_{1/2}(\norm{\svec_1}_2)} \leq \norm{\svec_1}_2^{-1}$. Then we have
	\begin{align*}
	\norm{T^{-1}\xvec}_\infty 
	&\le  \gamma^{-1/2}( \kappa+1)^2 \biggl(\Dmax^{1/2}\rho+\frac{\norm{\svec_1}_\infty}
	{\norm{\svec_1}_2}\cdot \frac{\Dmax \norm{\svec}_1}{\Dmin^{1/2} \norm{\svec}_2}\rho\biggr)\\
	&\le \gamma^{-1/2}( \kappa+1)^2 \Dmax^{1/2}\, \rho \biggl(1+\frac{\Dmax\norm{\svec}_1\norm{\svec}_\infty}
	{\Dmin\norm{\svec}_2^2}\biggr)=\rho_2. 
	\end{align*}
	
	Finally, we prove part (a).  Define $X_3:=(I+X_2)^{-1}-I$.
	By \eqref{eq:transform} and since $T_1=T_1\trans$ we have $T_1\trans D^{-1/2} \bar A D^{-1/2} T_1=I+X_2$. Together with
	$T_1=(I+\svec_1\svec_1\trans)^{-1/2}$, this implies
	\[
	\bar A^{-1}= D^{-1/2}\, T_1\left(I+X_2\right)^{-1}\, T_1D^{-1/2} 
	  = D^{-1/2}\, T_1 X_3 T_1\trans\, D^{-1/2}+(D+\svec\svec\trans)^{-1}.
	\]
	By Lemma~\ref{diagonal}(a), $\maxnorm{X_3} \leq  \kappa( \kappa + 1)\gamma^{-1}n^{-1}$ and thus using \eqref{T1norms} we have
	\[
	\norm{\bar A^{-1}-(D+\svec\svec\trans)^{-1}}_{\infty}
	\le \Dmin^{-1}\, \norm{T_1}_1^2 \, \maxnorm{X_3} \le \frac{B^2  \kappa( \kappa + 1)}{\Dmin \,\gamma n}.
	\]
The expression for $(D+\svec\svec\trans)^{-1}$ follows from the Sherman--Morrison theorem
(see for example~\cite[equation (3.8.2)]{meyer}). 
\end{proof}

\begin{proof}[Proof of Lemma~\ref{lem:diaggeneral-r}]
	Define $\check\varLambda:=\varLambda(\betavec)$ and
	\[c:=\sqrt{\dfrac{1}{2}\check\varLambda \binom{n-2}{r-2}}.\]
	Then let $\svec:= (c,c, \ldots, c)\trans$
	and $D:=\operatorname{diag}(a_{11}-c^2,\ldots, a_{nn}-c^2)$.
	We write $A(\betavec) = D + \svec\svec\trans  + X$.
	
	First we show that the entries of $X$ are small. 
	Note that all the diagonal entries of $X$ are exactly 0. 
	By Lemma~\ref{lem:A-entries-tight}, the absolute value of any off-diagonal entry 
in $X$ is at most
	\begin{equation}\label{eq:xmax}
		|a_{jk}-c^2|\le (e^{4\deltab/r}-1)\check\varLambda \binom{n-2}{r-2}\stackrel{\eqref{eq:exp-r}}{\le} \frac{e^{4\deltab}}{r}\check\varLambda \binom{n-2}{r-2}.
	\end{equation}
	In addition, Lemma~\ref{lem:A-entries-tight} also implies that for any $1\le j \le n$ we have
	\begin{align}
		a_{jj}-c^2&\ge \dfrac{1}{2}e^{-4\deltab/r} \check\varLambda \binom{n-1}{r-1}
		    - \dfrac{1}{2}\check\varLambda\binom{n-2}{r-2}
		= \dfrac{1}{2}e^{-4\deltab/r} \check\varLambda \binom{n-1}{r-1} \left(1 -\frac{(r-1)e^{4\deltab/r}}{n-1}\right)\nonumber\\
		&\stackrel{\eqref{eq:exp-r}}{\ge} \dfrac{1}{2}e^{-4\deltab/r} \check\varLambda \binom{n-1}{r-1} \left(1 -\frac{(r-1)(e^{4\deltab}+r)}{r(n-1)}\right) 
		\ge \dfrac{1}{2}e^{-4\deltab/r} \check\varLambda \binom{n-1}{r-1} \left(1 -\frac{e^{4\deltab}+r}{n-1}\right)\nonumber\\
		&\ge \dfrac15 e^{-4\deltab/r} \check\varLambda \binom{n-1}{r-1},\label{eq:lower-bound-diag}
	\end{align}
	where in the last step we used $r\le n/2$ and $n\ge 16e^{4\deltab}$.
		
	Consider the value of $\gamma$ as in Lemma~\ref{motherL49}. For any $\yvec\in\Reals^n$ we have
	\begin{align*}
		\yvec\trans A(\betavec) \yvec
		&=\dfrac{1}{2}\sum_{W\in \rsets} \lambda_{W}(\betavec)(1-\lambda_{W}(\betavec))\biggl(\,\sum_{j\in W} y_j\biggr)^{\!2}
		\stackrel{L.\ref{lem:lambdaW-ratios}}{\geq}\dfrac{1}{2}e^{-2\deltab}\check\varLambda \sum_{W\in \rsets} \biggl(\,\sum_{j\in W} y_j\biggr)^{\!2} \\
		&=\dfrac{1}{2}e^{-2\deltab}\check\varLambda \yvec\trans\left(\binom{n-1}{r-1}I-\binom{n-2}{r-2}I+\binom{n-2}{r-2}J\right)\yvec\\
		&= \dfrac{1}{2}e^{-2\deltab}\check\varLambda\, \biggl( \binom{n-2}{r-1} \norm{\yvec}_2^2+\binom{n-2}{r-2} \biggl(\,\sum_{j=1}^n y_j\biggr)^{\!2}\,\biggr).
	\end{align*}
	On the other hand, by Lemma~\ref{lem:A-entries-tight} we have
	\begin{align*}
		\yvec\trans (D+s s\trans) \yvec &\le \dfrac{1}{2}e^{4\deltab} \check\varLambda\, \biggl(\binom{n-1}{r-1} \norm{\yvec}_2^2+ \binom{n-2}{r-2}\biggl(\sum_{j=1}^n y_j\biggr)^{\!2}\,\biggr)\\
		&= \dfrac{1}{2}e^{4\deltab} \check\varLambda\, \biggl(\frac{n-1}{n-r}\binom{n-2}{r-1} \norm{\yvec}_2^2+ \binom{n-2}{r-2}\biggl(\sum_{j=1}^n y_j\biggr)^{\!2}\,\biggr)\\  
		&\le 2 e^{4\deltab} \check\varLambda\, \biggl( \binom{n-2}{r-1} \norm{\yvec}_2^2+\binom{n-2}{r-2} \biggl(\sum_{j=1}^n y_j\biggr)^{\!2}\,\biggr),
	\end{align*}
	where the last inequality holds as $r\le n/2$. Therefore setting
	$\gamma:= e^{-6\deltab}/4$, we have
	 for any $\yvec\in\Reals^n$ that $\yvec\trans A\yvec \ge \gamma\, \yvec\trans (D+\svec \svec\trans) \yvec$.
	Let $B$ be as in Lemma~\ref{motherL49}. Then
	\begin{equation}\label{eq:orderB}
	B=1+\frac{\Dmax \norm{\svec}_1\norm{\svec}_{\infty}}{\Dmin \norm{\svec}_2^2}=1+\frac{\Dmax}{\Dmin}\le 4 e^{8\deltab/r},
	\end{equation}
	which follows from Lemma~\ref{lem:A-entries-tight} and \eqref{eq:lower-bound-diag}.
	
	For $\kappa$ as in Lemma~\ref{motherL49}, using \eqref{eq:xmax}, \eqref{eq:lower-bound-diag} and \eqref{eq:orderB}, we have
	\begin{equation}\label{eq:kappa}
		\kappa = B^2\, D_{\min}^{-1}\, n\, \maxnorm{X}
		\le 80  e^{16\deltab/r}  
		\frac{ e^{4\deltab} r^{-1} \check\varLambda\, \binom{n-2}{r-2}}{ e^{-4\deltab/r}\check\varLambda\, \binom{n-1}{r-1}}n 
		\le 80 e^{20\deltab/r+4\deltab}.
	\end{equation}
	
	Next we consider the matrix $(D+\svec\svec\trans)^{-1}$. By Lemma~\ref{motherL49} we have
	\[
	(D+\svec\svec\trans)^{-1}=D^{-1}-\frac{D^{-1}\svec \svec\trans D^{-1}}{1+\norm{D^{-1/2}\svec}_2^2},
	\]
	and we are interested in an upper bound on the absolute value of the elements of this matrix. First consider the vector $D^{-1} \svec$ and note that
	\[
	D^{-1} \svec=\begin{pmatrix} \frac{c}{a_{11}-c^2}\\ \vdots\\ \frac{c}{a_{nn}-c^2} \end{pmatrix}.
	\] 
	Together with \eqref{eq:lower-bound-diag} this implies that every element in the matrix $D^{-1}\svec \svec\trans D^{-1}$ has absolute value at most
	\begin{equation}\label{eq:numeratorg}
		16 e^{8\deltab/r}\frac{\check\varLambda \binom{n-2}{r-2}}{\check\varLambda^2\binom{n-1}{r-1}^2}\le  16 e^{8\deltab/r}\frac{1}{\check\varLambda \binom{n}{r}}.
	\end{equation}
	Similarly
	\[
	D^{-1/2} \svec=c\, \begin{pmatrix} (a_{11}-c^2)^{-1/2}\\ \vdots\\ (a_{nn}-c^2)^{-1/2}\end{pmatrix},
	\] 
	implying that
	\[
	\norm{D^{-1/2}\svec}_2^2=c^2\sum_{j=1}^n \frac{1}{a_{jj}-c^2}\stackrel{L.\ref{lem:A-entries-tight}}{\ge} e^{-4\deltab/r}\frac{\check\varLambda \binom{n-2}{r-2}}{\check\varLambda \binom{n-1}{r-1}}n\ge \frac{r}{2} e^{-4\deltab/r},
	\]
	and hence
	\begin{equation}\label{eq:Dsnorm}
		1+\norm{D^{-1/2}\svec}_2^2 \ge  \frac{r}{2}\, e^{-4\deltab/r}.
	\end{equation}
	Therefore, by \eqref{eq:numeratorg} and \eqref{eq:Dsnorm}, every element of $\frac{D^{-1}\svec \svec\trans D^{-1}}{1+\norm{D^{-1/2}\svec}_2^2}$ has absolute value at most
	\begin{equation}\label{eq:inverse-error}
		\frac{16}{1/2}\cdot  \frac{e^{12 \deltab/r}}{r \check\varLambda \binom{n}{r}}= 32 \frac{e^{12 \deltab/r}}{\check\varLambda \binom{n-1}{r-1}n},
	\end{equation}
	and thus so do the off-diagonal elements of $(D+\svec\svec\trans)^{-1}$. As for the diagonal elements, by \eqref{eq:lower-bound-diag} and \eqref{eq:inverse-error}, each has absolute value at most
	\[
	\frac{5e^{4 \deltab/r}}{\check\varLambda \binom{n-1}{r-1}}+\frac{32 e^{12 \deltab/r}}{ \check\varLambda \binom{n-1}{r-1} n} \le \frac{8 e^{12 \deltab/r}}{\check\varLambda \binom{n-1}{r-1}},
	\]
	as $n\ge 16e^{4\deltab}\ge 16$. 
	
	Now we have all the information needed to establish a bound on the absolute value of the elements in $A(\betavec)^{-1}$ using Lemma~\ref{motherL49}(a). In particular, using \eqref{eq:lower-bound-diag}, \eqref{eq:orderB} and \eqref{eq:kappa}, the diagonal entries of $A(\betavec)^{-1}$ have absolute value at most
	\[
	 \frac{8e^{12\deltab/r}}{\check\varLambda \binom{n-1}{r-1}}+\frac{B^2\kappa(\kappa+1)}
	{\Dmin\gamma n}
	\le  \frac{8 e^{12\deltab/r}}{\check\varLambda\binom{n-1}{r-1}}+\frac{\hat{C} e^{60\deltab/r+14\deltab}}{\check\varLambda \binom{n-1}{r-1}n}
	 \le (8+\hat{C}) \frac{e^{60\deltab/r+14\deltab}}{\check\varLambda \binom{n-1}{r-1}},
	\] 
	for some sufficiently large constant $\hat{C}$.
	On the other hand, the off-diagonal entries have absolute value at most
	\[
	 \frac{32 e^{12 \deltab/r}}{\check\varLambda \binom{n-1}{r-1}n}+\frac{B^2 \kappa(\kappa+1)}
	{\Dmin\gamma n}\le \frac{32 e^{12\deltab/r}}{\check\varLambda \binom{n-1}{r-1}n}+\frac{\hat{C} e^{60\deltab/r+14\deltab}}{\check\varLambda \binom{n-1}{r-1}n}
	\le (32+\hat{C})\left( \frac{e^{60\deltab/r+14\deltab}}{\check\varLambda \binom{n-1}{r-1}n}\right).
	\]
	The first statement follows by setting $C=32+\hat{C}$ and using the fact that $r\ge 3$.
	
	Now for the second statement. Substituting \eqref{eq:lower-bound-diag}, \eqref{eq:orderB} and \eqref{eq:kappa}  into  Lemma~\ref{motherL49}(b) gives
	\[
	\norm{T}_{1},\norm{T}_{\infty}=O\left(\frac{1}{\check\varLambda^{1/2}\binom{n-1}{r-1}^{1/2}}\right),
	\]
	as required.

	Now for the last statement of the lemma. For any real $z\ge 0$, let 
	\[
	\hat{\rho}(z)=z\frac{n}{r^{1/2}} \frac
	{\Dmin^{1/2}\norm{\svec}_2}{\Dmax\norm{\svec}_1}c\rho.
	\]
Then
\[ \calQ\(\hat{\rho}(z)\) =
\biggl\{\xvec\in U_n\(\hat{\rho}(z)\): \biggl|\, \sum_{j\in [n]} x_j \biggr|\le z\, n r^{-1/2} \rho \biggr\}.
\]
	Note that 
	\begin{align*}
		\frac{n}{r^{1/2}} \frac
		{\Dmin^{1/2}\norm{\svec}_2}{\Dmax\norm{\svec}_1}c
		&=\Theta \left(\frac{n}{r^{1/2}} \frac
		{\norm{\svec}_2}{\Dmin^{1/2}\norm{\svec}_1}c \right)
		= \Theta\left(  \frac{n}{r^{1/2}} \frac{1}{ \check\varLambda^{1/2} \binom{n-1}{r-1}^{1/2}}\frac{n^{1/2} c}{n c}c\right)\\
		&= \Theta\left(\frac{n^{1/2}}{r^{1/2}} \frac{1}{ \check\varLambda^{1/2} \binom{n-1}{r-1}^{1/2}}\check\varLambda^{1/2}\binom{n-2}{r-2}^{\!1/2}\,\right)
		= \Theta\biggl(\biggl(\frac{n(r-1)}{(n-1)r}\biggr)^{\!\!1/2}\,\biggr)=\Theta (1).
	\end{align*}
	Therefore there exists $z_1=\Omega(1)$ such that $\hat{\rho}(z_1)\le \rho$ and $z_1\leq 1$.
	Together with Lemma~\ref{motherL49}, this implies that
	\[
	T\( U_n(\rho_1)\) \subseteq
	\calQ\(\hat{\rho}(z_1)\) \subseteq \calR(\rho),
	\]
	where 
	\[
	\rho_1 =  \dfrac{1}{B} \, \Dmin^{1/2}\, \gamma^{1/2}\, (1 + \kappa)^{-1} \hat{\rho}(z_1)=\Theta\biggl(\check\varLambda^{1/2}\binom{n-1}{r-1}^{\!1/2}\rho\biggr).
	\]	
	Similarly there must exist $z_2=O(1)$ such that $\hat{\rho}(z_2)\ge \rho$ and $z_2\geq 1$.
	Then by Lemma~\ref{motherL49} we have
	\[
	T\( U_n(\rho_2)\) \supseteq
	\calQ\(\hat{\rho}(z_2)\) \supseteq \calR(\rho),
	\]
	where 
	\[\rho_2 = B \Dmax^{\, 1/2}\, \gamma^{-1/2}\, (1 + \kappa)^2\hat{\rho}(z_2)
	=\Theta\biggl(\check\varLambda^{1/2}\binom{n-1}{r-1}^{\!1/2}\rho\biggr).
	\]	
This completes the proof.	
\end{proof}


\bigskip

\section{Appendix: useful identities}\label{appendix}

In this appendix we provide summations that help for the calculations in Section~\ref{s:nearreg}.
We use the notation
\[
N = \binom{n-1}{r-1} = \frac d\lambda, \quad
\W_s = \W_s(W) = \sum_{\ell\in W} \delta_\ell^s, \quad
R_s = \sum_{\ell=1}^n \delta_\ell^s
\]
and recall that $R_1=0$.
We provide approximations for some expressions, assuming that
$(r,\dvec)$ belongs to the first quadrant and $\delta_{\max} = O(d^{3/5})$.
The error bounds are good enough for our applications but are not necessarily tight.

\subsection{Summations over all $W\in\rsets$}\label{appendix:all}
\def\suma{\frac 1N\sum_{W\in\rsets}}
\begin{align*}
	\suma \W_\ell &= R_\ell\hspace*{104mm} (\ell\ge 1),  \\
	\suma \W_1\W_\ell &= {\frac {(n-r)R_{\ell+1}}{n-1}} 
\hspace*{88mm} (\ell\ge 1), \\
	\suma \W_1^3 &= {\frac {(n-r)(n-2 r)R_3}{(n-2)(n-1)}}
  = R_3 + O\(\delta_{\max}\, d^{7/5}\), \\
	\suma \W_1^4 &= {\frac {3(r-1)(n-r)(n-r-1)R_2^2}{(n-3)(n-2)(n-1)}}+{\frac {(n-r)(n^2- 6 r n+6 r^2 +n)R_4}{(n-3)(n-2)(n-1)}} \\
  &= \frac{3(r-1) R_2^2}{n} + R_4 + O\( \delta_{max}\, d^{12/5}\),\\
	\suma \W_2^2 &= {\frac {(r-1) R_2^2}{n-1}} + {\frac {(n-r)R_4}{n-1}}
         = \frac{(r-1)R_2^2}{n} + R_4 + O\(\delta_{\max}\, d^{12/5}\), \\
	\suma \W_1^2\, \W_2 &= 
    {\frac {(r-1)(n-r)R_2^2}{(n-2)(n-1)}}+{\frac {(n-r)(n-2 r)R_4}{(n-2)(n-1)}}\\
  &= \frac{(r-1)R_2^2}{n} + R_4 + O\( \delta_{\max}\, d^{12/5}\).
\end{align*}

\subsection{Summations over all $W\ni j$}\label{appendix:j}
\def\sumj{\frac 1N\sum_{W\ni j}}
\begin{align*}
	\sumj \W_\ell &= {\frac {(r-1) R_\ell}{n-1}}+{\frac {(n-r)\, \delta_j^\ell}{n-1}} \hspace*{8cm} (\ell\ge 1), \\
	\sumj \W_1 \W_\ell &= {\frac {(r-1)  (n-r)\delta_j R_\ell}{(n-2)(n-1)}}+
{\frac {(r-1)(n-r) R_{\ell+1}}{(n-2)(n-1)}}+{\frac { (n-r)(n-2 r)\delta_j^{\ell+1}}{(n-2)(n-1)}} 
\hspace*{6mm} (\ell\ge 1),\\
	\sumj \W_1^3 &= {\frac {3(r-1)  (n-r)(n-r-1)\delta_j R_2}{(n-3)(n-2)(n-1)}} +{\frac {(r-1)(n- r)(n-2 r+1)R_3}{(n-3)(n-2)(n-1)}} \\
	&{\qquad} {} +{\frac { (n-r)(n^2-6 r n+6 r^2 +n)\delta_j^3}{(n-3)(n-2)(n-1)}} \\
  &= \frac{3(r-1)\delta_j R_2 +(r-1)R_3}{n} + \delta_j^3 + O\left(\frac{d^{12/5}}{rn}\right), \\
	\sumj \W_1^4 &= {\frac {3(r-2)(r-1)(n-r)(n-r-1)R_2^2}{(n-4)(n-3)(n-2)(n-1)}} \\
	&{\qquad} +{\frac {6(r-1) (n-r)(n-r-1)(n-2 r)\delta_j^2 R_2}{(n-4)(n-3)(n-2)(n-1)}} \\
	&{\qquad} + {\frac {4(r-1) (n-r)(n-r-1)(n-2 r)\delta_j R_3}{(n-4)(n-3)(n-2)(n-1)}} \\
	&{\qquad} +{\frac {(r-1)(n-r)(n^2-6 r n+6 r^2 +5 n-6 r)R_4}{(n-4)(n-3)(n-2)(n-1)}} \\
	&{\qquad}+{\frac { (n-r)(n-2 r)(n^2-12 r n+12 r^2 +5 n)\delta_j^4}{(n-4)(n-3)(n-2)(n-1)}}\\
  &= O\left(\frac{d^{17/5}}{rn}\right). 
\end{align*}

\subsection{Summations over all $W\supset \{j,k\}$}\label{appendix:jk}
\def\sumjk{\frac 1N\sum_{W\supset \{j,k\}}}
\begin{align*}
	\sumjk \W_\ell &= {\frac {(r-2)(r-1) R_\ell}{(n- 2)(n-1)}}+{\frac {(r-1)(n-r)(\delta_j^\ell +\delta_k^\ell)}{(n-2)(n-1)}} \\
   &= \frac{(r-2)(r-1)R_\ell}{n^2} + \frac{(r-1)(\delta_j^\ell + \delta_k^\ell)}{n}
   + O\left(\frac{\delta_{\max} r\, d^{\ell-3/5}}{n^2}\right) \hspace*{1cm}(\ell\ge 1), \\
	\sumjk \W_1^2 &= {\frac {(r-2)(r-1)(n-r)R_2}{(n-3)(n-2)(n-1)}} \\
	&\quad {} +{\frac {(r-1)(n-r)\((n-2r+1)(\delta_j^2 + \delta_k^2)  +2(n-r-1) \delta_j \delta_k\)}{(n-3)(n-2)(n-1)}}\\
  &= \frac{(r-2)(r-1)R_2}{n^2} + \frac{(r-1)\( \delta_j + \delta_k\)^2}{n} 
   + O\left(\frac{\delta_{\max} r\, d^{7/5}}{n^2}\right).
\end{align*}

\subsection*{Acknowledgement}
We would like to thank the anonymous referee for their helpful comments. 

\nicebreak


\begin{thebibliography}{99}

\itemsep=0pt

\bibitem{Bishop}
Y.\,M.~Bishop, S.\,E.~Fienberg and P.\,W.~Holland,
\textit{Discrete Multivariate Analysis: Theory and Applications},
Springer, Berlin, 2007.

\bibitem{BG}
V.~Blinovsky and C.~Greenhill,
Asymptotic enumeration of sparse uniform hypergraphs with
given degrees, 
\textit{European J.~Combin.} {\bf 51} (2016), 287--296.

\bibitem{BG2}
V.~Blinovsky and C.~Greenhill,
Asymptotic enumeration of sparse uniform linear hypergraphs
with given degrees, 
\textit{Electron.\ J.~Combin.} {\bf 23(3)} (2016), \#P3.17.

\bibitem{CGM}
E.\,R.~Canfield, C.~Greenhill and B.\,D.~McKay, 
Asymptotic enumeration of dense 0-1 matrices with specified line sums, 
\emph{Journal of Combinatorial Theory (Series A)} {\bf 115} (2008), 
32--66.

\bibitem{correlation-immune}
E.\ R.~Canfield, Z.~Gao, C.~Greenhill, B.\ D.~McKay and R.\ W.~Robinson, 
Asymptotic enumeration of correlation-immune boolean functions, 
\emph{Cryptography and Communications} {\bf 2} (2010), 111--126.

\bibitem{Chat2011}
S.~Chatterjee, P.~Diaconis, and A.~Sly. Random graphs with a given degree
sequence. \emph{Ann.\ Appl.\ Probab.}, 21(4):1400--1435, (2011).

\bibitem{DFRS}
A.~Dudek, A.~Frieze, A.~Ruci{\' n}ski and M.~{\v S}ileikis,
Approximate counting of regular hypergraphs,
\textit{Inform. Process. Lett.} {\bf 113} (2013), 19--21.

\bibitem{FKM}
N.~Fountoulakis, M.~Kang and T.~Makai,
Resolution of a conjecture on majority dynamics: Rapid stabilization
in dense random graphs, \textit{Random Structures Algorithms},
\textbf{57} (2020) 1134--1156.

\bibitem{sandwichGIM}
P. Gao, M. Isaev and B.\,D. McKay,
   Sandwiching random regular graphs between binomial random graphs,
   SIAM Conference on Discrete Algorithms (SODA 2020), pp.~690--701. 

\bibitem{golubski}
A.\,J.~Golubski, E.\,E.~Westlund, J.~Vandermeer and M.~Pascual,
Ecological networks over the edge: hypergraph trait-mediated indirect
interaction (TMII) structure,
\textit{Trends in Ecology \& Evolution} {\bf 31} (2016), 344--354.

\bibitem{Higham}
N.\,J.~Higham, \textit{Functions of Matrices},
Society for Industrial and Applied Mathematics, 2008. 

\bibitem{HJ}
R.\,A.~Horn and C.\,R.~Johnson, \emph{Matrix Analysis (2nd edn.)},
Cambridge University Press, Cambridge, 2013.

\bibitem{mother}
M.~Isaev and B.\,D.~McKay,
Complex martingales and asymptotic enumeration,
\textit{Random Structures Algorithms} {\bf 52} (2018), 617--661.


\bibitem{JansonRG}
S.~Janson, T.~\L uczak and A.~Rucinski,
\textit{Random Graphs}, John Wiley \&\ Sons, New York, 2000.

\bibitem{KLW}
N.~Kam{\v c}ev, A.~Liebenau and N.~Wormald, 
Asymptotic enumeration of hypergraphs by degree sequence,
\emph{Advances in Combinatorics}, 2022:1, 33pp.

\bibitem{KS}
M.\,G.~Kendall and A.~Stuart,
\textit{The Advanced Theory of Statistics, vol.~1},
Charles Griffin \& Company, London, 1958.

\bibitem{Kuperberg}
G. Kuperberg, S. Lovett and R. Peled, Probabilistic existence of regular
combinatorial structures, \textit{Geom.\ Funct.\ Anal.} {\bf 27} (2017), 919--972.

\bibitem{LW1} 
A.~Liebenau and N.~Wormald,
Asymptotic enumeration of graphs by degree sequence, and the
degree sequence of a random graph. 
Preprint, 2017. \texttt{arXiv:1702.08373}

\bibitem{LW2}
A.~Liebenau and N.~Wormald,
Asymptotic enumeration of digraphs and bipartite graphs by degree sequence. 
\textit{Random Structures Algorithms}, to appear (2022).

\bibitem{MM}
B.\, D.~McKay and J.\ C.~McLeod, 
Asymptotic enumeration of symmetric integer matrices with uniform row sums, 
\emph{Journal of the Australian Mathematical Society} {\bf 92(3)} (2012),
367--384.

\bibitem{McKW90}
B.\,D.~McKay and N.\,C.~Wormald, Asymptotic enumeration by degree sequence of
graphs of high degree, \emph{European J.\ Combin.} {\bf 11} (1990), 565--580.

\bibitem{meyer}
C.\,D.~Meyer, \textit{Matrix Analysis and Applied Linear Algebra},
Society for Industrial and Applied Mathematics, Philadelphia, 2000.

\bibitem{MNBO}
J.\,V.~Michalowicz, J.\,M.~Nichols, F.~Bucholtz and C.\,C.~Olson,
An Isserlis' theorem for mixed Gaussian variables: application to the auto-bispectral
density, \emph{J.\ Stat.\ Phys.} {136} (2009), 89--102.

\bibitem{morimae}
T.~Morimae, Y.~Takeuchi and M.~Hayashi,
Verification of hypergraph states, \textit{Phys. Rev. A} {\bf 96} (2017), 062321.

\bibitem{PP}
 R.~Pemantle and Y.~Peres,
 Concentration of Lipschitz functionals of determinantal and other
 strong Rayleigh measures,
 \textit{Combin. Prob. Comput.}, \textbf{23} (2014) 140--160.

\bibitem{purkait}
P.~Purkait, T.-J.~Chin, A.~Sadri and D.~Suter, Clustering with
hypergraphs: the case for large hyperedges,
\textit{IEEE Transactions on Pattern Analysis and
Machine Learning} {\bf 39} (2017), 1697--1711.

\bibitem{stasi}
D.~Stasi, K.~Sadeghi, A.~Rinaldo, S.~Petrovi{\' c} and S.\,E.~Fienberg, 
$\beta$-models for random hypergraphs with a given degree sequence,
in \textit{Proceedings of the 21st International Conference on Computational
Statistics (COMPSTAT 2014)} (M.~Gilli, G.~Gonzalez-Rodriguez and A.~Nieto-Reyes,
eds.), Curran Associates, New York, 2014, pp.~593--600.

\bibitem{VM}
 V.\,A.~Vatutin and V.\,G.~Mikhailov,
 Limit theorems for the number of empty cells in an equiprobable
 scheme for group allocation of particles,
 \textit{Theory Probab. Appl.}, \textbf{27} (1982) 734--743.


\bibitem{W-survey}
N.\,C.~Wormald, Asymptotic enumeration of graphs with given degree sequence,
in \textit{Proceedings of the International Congress of Mathematicians, ICM 2018} 
(B.~Sirakov, P.\,N.~de Souza and M.~Viana, eds.), vol.~4,
World Scientific, pp.~3263--3282.

\bibitem{Zhan2002}
X. Zhan, \textit{Matrix Inequalities}, Lecture Notes in Mathematics, Vol. 1790.
Springer, Berlin, 2002.

\end{thebibliography}
\end{document}